\documentclass[12pt]{article}
\usepackage{amsmath}
\usepackage{amssymb}
\usepackage{bbm}
 \usepackage[hmargin=3.54cm,tmargin=3.7cm,bmargin=3.7cm]{geometry}

\numberwithin{equation}{section}

\usepackage{graphicx,psfrag}%

\usepackage{graphicx}

\usepackage{pict2e}%

\centerline{\LARGE A PDE for non-intersecting Brownian motions}
\centerline{\LARGE  and applications}
\bigskip
\centerline{Mark Adler\footnote{Department of Mathematics, Brandeis University, Waltham, Mass 02454, USA,
adler@brandeis.edu. The support of a National Science Foundation grant \# DMS-07-04271 is gratefully acknowledged},
Jonathan Del\'epine\footnote{D\'epartement de Math\'ematiques, Universit\'e Catholique de Louvain, 1348
Louvain-la-Neuve, Belgium, jonathan.delepine@uclouvain.be}, Pierre van Moerbeke\footnote{D\'epartement de Math\'ematiques, Universit\'e Catholique
de Louvain, 1348 Louvain-la-Neuve, Belgium and Brandeis University, Waltham, Mass 02454, USA,
pierre.vanmoerbeke@uclouvain.be. The support of a National Science Foundation grant \# DMS-07-04271, a European Science
Foundation grant (MISGAM), a Marie Curie Grant (ENIGMA), Nato, FNRS and Francqui Foundation grants is gratefully
acknowledged.} and Pol Vanhaecke\footnote{The support of a European Science Foundation grant (MISGAM), a Marie
Curie Grant (ENIGMA) and a ANR Grant (GIMP) is gratefully acknowledged.}}

\newcommand{\MAT}[1]{\left(\begin{array}{*#1c}}
\newcommand{\mat}{\end{array}\right)}
\newcommand{\qed}{\leavevmode\unskip\nobreak\penalty200\hskip2pt\null
\nobreak\hfill\rule{1.1ex}{1.1ex}
\medbreak }

\def\ds{\displaystyle}

\newcommand{\Id}{\mathbbm{1}}
\newcommand{\mi}{{-1}}
\newcommand{\rg}{\rightarrow}

\newcommand{\JR}{{\cal J}}
\newcommand{\J}[2]{{\cal J}_{#1#2}}
\newcommand{\LR}{{\cal L}}

\newcommand{\PR}{{\cal P}}

\newcommand{\BB}{{\cal B}}

\newcommand{\BP}{{\mathbb P}}
\newcommand{\BQ}{{\mathbb Q}}

\newcommand{\BX}{{\mathbb X}}
\newcommand{\BY}{{\mathbb Y}}
\newcommand{\BZ}{{\mathbb Z}}
\newcommand{\Sg}[2]{\Sigma^{(#1)}_{#2}}
\newcommand{\Tg}[2]{T^{(#1)}_{#2}}
\newcommand{\hTg}[2]{\hat T^{(#1)}_{#2}}
\newcommand{\iy}{\infty}
\newcommand{\pl}{\partial}
\newcommand{\al}{\alpha}
\renewcommand{\a}{\alpha}

\newcommand{\no}{\nonumber}

\newenvironment{proof}{\medskip\noindent{\it Proof:\/} }{\qed}

\newcommand{\la}{\langle}
\newcommand{\ra}{\rangle}

\newcommand{\dt}{\delta}
\newcommand{\Dt}{\Delta}
\newcommand{\vr}{\varepsilon}
\newcommand{\sg}{\sigma}
\newcommand{\BR}{{\mathbb R}}
\newcommand{\BE}{{\mathbb E}}

\newcommand{\tr}{\mbox{tr}}

\newcommand{\rr}{\operatorname{tr}}

\newcommand{\spec}{\operatorname{spec}}
\newcommand{\diag}{\operatorname{diag}}

\def\be#1\ee{\begin{equation}#1\end{equation}}
\def\bea#1\eea{\begin{eqnarray}#1\end{eqnarray}}
\def\bean#1\eean{\begin{eqnarray*}#1\end{eqnarray*}}

\newtheorem{definition}{Definition}[section]
\newtheorem{theorem}[definition]{Theorem}
\newtheorem{lemma}[definition]{Lemma}
\newtheorem{corollary}[definition]{Corollary}
\newtheorem{proposition}[definition]{Proposition}
\newtheorem{remark}[definition]{Remark}
\newtheorem{conjecture}[definition]{Conjecture}

\catcode `!=11

\newdimen\squaresize
\newdimen\thickness
\newdimen\Thickness
\newdimen\ll! \newdimen \uu! \newdimen\dd! \newdimen \rr! \newdimen
\temp!

\def\sq!#1#2#3#4#5{%
\ll!=#1 \uu!=#2 \dd!=#3 \rr!=#4
\setbox0=\hbox{%
 \temp!=\squaresize\advance\temp! by .5\uu!
 \rlap{\kern -.5\ll!
 \vbox{\hrule height \temp! width#1 depth .5\dd!}}%
 \temp!=\squaresize\advance\temp! by -.5\uu!
 \rlap{\raise\temp!
 \vbox{\hrule height #2 width \squaresize}}%
 \rlap{\raise -.5\dd!
 \vbox{\hrule height #3 width \squaresize}}%
 \temp!=\squaresize\advance\temp! by .5\uu!
 \rlap{\kern \squaresize \kern-.5\rr!
 \vbox{\hrule height \temp! width#4 depth .5\dd!}}%
 \rlap{\kern .5\squaresize\raise .5\squaresize
 \vbox to 0pt{\vss\hbox to 0pt{\hss $#5$\hss}\vss}}%
}
 \ht0=0pt \dp0=0pt \box0
}

\def\vsq!#1#2#3#4#5\endvsq!{\vbox to \squaresize{\hrule
width\squaresize height 0pt%
\vss\sq!{#1}{#2}{#3}{#4}{#5}}}

\newdimen \LL! \newdimen \UU! \newdimen \DD! \newdimen \RR!

\def\vvsq!{\futurelet\next\vvvsq!}
\def\vvvsq!{\relax
  \ifx     \next l\LL!=\Thickness \let\continue=\skipnexttoken!
  \else\ifx\next u\UU!=\Thickness \let\continue=\skipnexttoken!
  \else\ifx\next d\DD!=\Thickness \let\continue=\skipnexttoken!
  \else\ifx\next r\RR!=\Thickness \let\continue=\skipnexttoken!
  \else\def\continue{\vsq!\LL!\UU!\DD!\RR!}%
  \fi\fi\fi\fi
  \continue}

\def\skipnexttoken!#1{\vvsq!}

\def\place#1#2#3{\vbox to 0pt{\vss
\rlap{\kern#1\squaresize
  \raise#2\squaresize\hbox{$#3$}}
\vss}}

\catcode `!=12

\renewcommand{\(}{\left(}
\renewcommand{\)}{\right)}
\newcommand{\geqs}{\geqslant}
\newcommand{\leqs}{\leqslant}
\newcommand{\s}[2]{s^{(#1)}_{#2}}
\newcommand{\x}[2]{{x^{(#1)}_{#2}}}
\renewcommand{\u}[2]{{u^{(#1)}_{#2}}}
\renewcommand{\v}[2]{{v^{(#1)}_{#2}}}
\newcommand{\y}[2]{{y^{(#1)}_{#2}}}
\renewcommand{\rq}{{(r,q)>(1,1)}}
\renewcommand{\b}[2]{b^{(#1)}_{#2}}
\newcommand{\hB}[2]{{\hat\BB}^{(#1)}_{#2}}
\newcommand{\B}[2]{{\BB}^{(#1)}_{#2}}
\newcommand{\pps}[2]{\pp{}{s^{(#1)}_{#2}}}
\newcommand{\ppb}[2]{\pp{}{b^{(#1)}_{#2}}}
\newcommand{\ppc}[1]{\pp{}{c_{#1}}}
\newcommand{\pp}[2]{\frac{\partial #1}{\partial #2}}
\newcommand{\p}[2]{\partial^{(#1)}_{#2}}

\newcommand{\cons}{{\sum_1^{p}\kappa_\ell\b \ell {1,2}=0}}
\newcommand{\prodindex}{{\renewcommand{\arraystretch}{0.6}
    \begin{array}{c} 
      \scriptstyle 1\leqs i\leqs  N\\ 
      \scriptstyle 1\leqs k\leqs m
    \end{array}}}

\begin{document}


\begin{abstract}

Consider 
$N=n_1+n_2+\cdots+n_p$ non-intersecting Brownian motions on the real line,
starting from the origin at $t=0$, with $n_i$ particles forced to reach $p$ distinct target points $\beta_i$ at
time $t=1$, with $\beta_1<\beta_2<\cdots<\beta_p$. This can be viewed as a diffusion process in a sector of $\BR^N$. This work shows that the transition probability, that is the probability for the particles to pass
through windows $\tilde E_k$ at times $t_k$, satisfies, in a new set of variables, a non-linear PDE which can be expressed as a near-Wronskian; that is a determinant of a matrix of size $p+1$, with each row being a derivative of the previous, except for the last column. It is an interesting open question to understand those equations from a more probabilistic point of view. 

As an application of these equations, let the number of particles forced to the extreme points $\beta_1$ and $\beta_p$  tend to infinity; keep the number of particles forced to intermediate points fixed (inliers), but let the target points themselves go to infinity according to a proper scale. A new critical process appears at the point of bifurcation, where the bulk of the particles forced to $-\sqrt{n}$ depart
from those going to $\sqrt{n}$. These statistical fluctuations near that point of bifurcation are specified by a kernel, which is a rational perturbation of the Pearcey kernel. This work also shows that such equations are an essential tool in obtaining certain asymptotic results. Finally, the paper contains a conjecture.

\end{abstract}

 \newpage

\tableofcontents


\section{Introduction}

Consider $N$ non-intersecting Brownian motions $x_1(t)<x_2(t)<\ldots<x_N(t)$ on $\BR$ (Dyson's Brownian motions),
all starting at source points $\gamma_1<\gamma_2<\cdots<\gamma_N$ at time $t=0$ and forced to target points
$\delta_1<\delta_2<\cdots<\delta_N$ at $t=1$. According to the Karlin-McGregor formula \cite{KM}, the probability
that the $N$ particles pass through the subsets $\tilde E_1,~\tilde E_2,\ldots , ~\tilde E_m\subset \BR$
respectively at times $ 0<t_1<t_2<\cdots<t_m<1$ is given by (setting $t_0:=0$ and $t_{m+1}:=1$),
%
%
%
$$
\BP \left(
 \displaystyle \bigcap_{k=1}^m\left\{ \mbox{all}~ x_i(t_k)\in \tilde E_k\right\}
  \begin{tabular}{c|l}
    &\ $x_j(0)=\gamma_j,~~x_j(1)=\delta_j$,  \\
    &\ \mbox{for}~$j=1,\dots,N$
  \end{tabular}\right)
  $$
$$
=\frac1{Z_N}\int_{\tilde E_1^N}\prod _{i=1}^N du_i^{(1)}
\int_{\tilde E_2^N}\prod _{i=1}^N du_i^{(2)}
\ldots
\int_{\tilde E_m^N}\prod _{i=1}^N du_i^{(m)}
 \det(p(t_1-t_0;\gamma_i,u^{(1)}_j))_{1\leqs i,j\leqs N}
$$
\be
\times \det(p(t_2-t_1;u^{(1)}_i,u^{(2)}_j))_{1\leqs i,j\leqs N}
\ldots
\det(p(t_{m+1}-t_m;u^{(m)}_i,\dt_j))_{1\leqs i,j\leqs N}
 \label{tr-pr}\ee
where $ p(t,x,y)$ denotes the standard Brownian transition probability,
\be
  p(t,x,y):= \frac{1}{\sqrt{ \pi t} }~e^{-\frac{(y-x)^2}{ t}}. \label{BrownianTransition}
\ee
%
There has been a great deal of interest in non-intersecting Brownian motions and especially in some critical
infinite-dimensional diffusions arising when the number of particles $N\rg \iy$. This in turn has been motivated by
random matrix theory and Dyson's observation \cite{Dyson} that letting the entries of GUE matrices run according to independent
Ornstein-Uhlenbeck processes leads to such non-intersecting Brownian motions for the random eigenvalues of the
matrix.

When some source points and some target points coincide, the formula (\ref{tr-pr}) for the probability must be
adapted by taking appropriate limits; see \cite{KM,Johansson1,BleKui3,AMPearcey}. In this paper, we consider the situation where the source points all coincide
with $0$, while some target points may coincide. Consider thus $N=n_1+n_2+\cdots+n_p$ non-intersecting Brownian
motions 
starting from the origin at $t=0$, with $n_i$ particles forced to reach $p$ distinct target points $\beta_i$ at
time $t=1$, with $\beta_1<\beta_2<\cdots<\beta_p$ in $\BR$; see Figure \ref{fig1}.

Given positive integers $n=(n_1,\ldots,n_p)$, given $m$ subsets $\tilde E_1,\dots,\tilde E_m\subset\BR$ and times
$t_0=0<t_1<t_2<\cdots<t_m<t_{m+1}=1$, this paper deals with the



\noindent {\em probability}\footnote{$\tilde\BE=\tilde E_1\times\ldots\times\tilde E_m.$}
$\BP^{(\beta)}_n(t,\tilde\BE)$, as in (\ref{prob-model}) below (i.e., the probability for the particles to pass
through the windows $\tilde E_k$ at times $t_k$); as is well-known, (see \cite{Zinn2,BleKui3,TW-Pear,AMPearcey}), $\BP^{(\beta)}_n(t,\tilde\BE)$ can also be
viewed as the probability for the eigenvalues of a chain of $m$ coupled Hermitian random matrices, after some change of
variables:
\bea 
 \BP^{(\beta)}_n(t,\tilde\BE)&:=&\BP\left({\displaystyle\bigcap_{k=1}^m\left\{\mbox{all}~x_i(t_k)\in\tilde E_k\right\} }
 \Bigr| 
  \begin{array}{l}
    \mbox{all $x_i(0)=0;$}\\
     \mbox{$n_j$ paths end up at $\beta_j$ at $t=1$,}\\
     \mbox{for $1\leqs j\leqs p$}
\end{array}\right) \no\\
&=&\frac1{\tilde Z_n}\int_{\spec(M_k)\in E_k}e^{-\frac12\tr
       \left(\sum\limits_{k=1}^mM_k^2-2\sum\limits_{k=1}^{m-1}c_kM_kM_{k+1}-2AM_m\right)}\prod_{k=1}^m dM_k
\no\\
&=:&\BP^A_n(c,\BE). \label{prob-model}\eea
The change of variables is given by the following formulae\footnote{\label{foot6}For $m=1$, the matrix integral
  above becomes a one-matrix integral with external potential. The change of variables below becomes:
$ 
  b_\ell=\sqrt{\frac{2t}{1-t }}\beta_\ell,~~
 E=
 \tilde E\sqrt{\frac{2 }{t(1-t)}}.
$ 
},
$$
 A:= \diag(\overbrace{b_1,\dots,b_1}^{n_1},\overbrace{b_2,\dots,b_2}^{n_2},\dots,\overbrace{b_p,\dots,b_p}^{n_p})
  ,~ \mbox{with}~b_\ell=\sqrt{\frac{2(t_m-t_{m-1})}{(1\!-\!t_m)(1\!-\!t_{m-1})}}\beta_\ell $$
\be
E_k: 
=\tilde E_k\sqrt{\frac{2(t_{k+1}-t_{k-1})}{(t_k-t_{k-1})(t_{k+1}-t_k)}}
,~~~
c_k^2:=\frac{ (t_{k+2}-t_{k+1}) (t_{k}-t_{k-1})}{(t_{k+2}-t_{k})(t_{k+1}-t_{k-1})}
,\label{change of variables}\ee
for $\ell=1,\dots,p$ and $k=1,\dots,m$. It is quite natural to impose a {\em linear constraint on the rescaled
target points} $\beta_1,\ldots , \beta_p$, namely 
\be {\sum_{\ell=1}^{p}}\kappa_\ell \beta_\ell=0,~\mbox{with}~{\sum_{\ell=1}^{p}}\kappa_\ell=1,~~\mbox{set}~
\kappa_0:=-1. \label{constraint} 
\ee
Of course, the same relation holds for the $b_i$'s. 
For instance, a typical situation is to take $\beta_1=-\beta_p$ and have all the remaining target points in arbitrary
position between $\beta_1$ and $\beta_p$. This case will be discussed in Section \ref{section:example}.

The natural initial or rather ``{\em final condition}" for the transition probability (\ref{prob-model}) is given by what happens when $t_m\rg 1$, keeping $t_1,\dots,t_{m-1}$, away from $0$ or $1$; namely, 
\be
  \lim_{t_m\rg 1} \BP^{(\beta)}_n(t,\tilde\BE)
 = 0\mbox{, when}~\tilde E_m ~  /  \!\!\!\!\!\supset \{\beta_1,\ldots,\beta_p\}.
\label{fc}  \ee

It is also known (see (\cite{TW-Pear})) that the probability above $\BP^{(\beta)}_n(t_1,\ldots,t_{m},\tilde E_1\times\ldots\times\tilde E_m)=\det(\Id-\raisebox{1mm}{$\chi$}{}_{_{\widetilde E_i^c}}(x)
        H^{(N)}_{t_it_j}(x,y)\raisebox{1mm}{$\chi$}{}_{_{\widetilde E_j^c}}(y))$ can be expressed as a matrix Fredholm determinant of a matrix kernel 
\footnote{\label{footMatrixFred}
  The Fredholm determinant of a matrix kernel $\widehat H_{t_it_j}(x,y):=\chi_{E_i}(x)H_{t_it_j}(x,y)\chi_{E_j}(y)$:
 \bean && \hspace*{-.8cm} {  \det
  \left( I-z (\widehat H
  _{t_it_j})_{1\leqs i,j\leqs m}
   \right) 
} 
         \\
%
%
  &&
  ~~~~~
 \hspace*{-1.5cm}= \!1\!+\! \sum_{n=1}^{\iy}{(-z)^n}
 \!\!\!\! \sum_{{0\leqs r_i\leqs n}\atop{\sum_1^mr_i=n}} {\int_{\cal R} }~~\prod_1^{r_1} d\al_i^{(1)}
   \ldots
   \prod_1^{r_m} d\al_i^{(m)}
   \det\left( \Bigl(
   \widehat H
    _{t_kt_{\ell}}(\al_i^{(k)},\al_j^{(\ell)})
    \Bigr)_{ {1\leqs i \leqs r_k}\atop {1\leqs j \leqs r_{\ell}}  }
 \right)_{1\leqs k,\ell\leqs m} ,
\eean
where the $n$-fold integral in each term above is taken over the
range
$$\cal R=
 { \left\{\begin{array}{c}
          -\iy <\al_1^{(1)}\leqs \ldots\leqs
 \al_{r_1}^{(1)}<\iy\\
           \vdots\\
     -\iy <\al_1^{(m)}\leqs \ldots\leqs \al_{r_m}^{(m)}<\iy\\
         \end{array}\right\} }.$$
}
 \bean
%
H_{t_k,t_\ell}^{(N)}(x,y; \beta_1,\ldots,\beta_p)dy &=&
 -\frac{dy}{2 \pi^2 \sqrt{(1\!-\!t_k)(1\!-\!t_\ell)}}
  \int_{\mathcal{C}} dV
\int_{\Gamma_{L}} dU ~
\frac{e^{ -\frac{ t_kV^2}{1-t_k}  +
\frac{2xV}{1-t_k} }}
{e^{ \frac{-t_\ell U^2}{1-t_\ell} +\frac{2yU}{1-t_\ell} }}
\no\\ 
&&~~~~~~ \times
\prod_{r=1}^p\bigg(\frac{U-\beta_r}{V-\beta_r}\bigg)^{n_r}
\frac{1}{U-V} \no
\eean

\bean &&-\left\{ \begin{array}{l}
            0,\,~~~~~~~~~~~~~~~~~~~~~~~~~~~~~~~~~~~~~
            \mbox{for}~~t_k\geqs
            t_\ell\\  
             \frac{dy}{\sqrt{\pi (t_\ell-t_k)}}
            e^{-\frac{(x-y)^2}{t_\ell-t_k}}
            e^{\frac{x^2}{1-t_k}-\frac{y^2}{1-t_\ell}},~~~~~
            \mbox{for}~~t_k<t_\ell
            \end{array}\right.    
\eean
\be \label{BM kernel}
\ee
where ${\cal C}$ is a closed contour enclosing all the points $\beta_r$, which is to the left of the line
$\Gamma_L:=L+i\BR$ by picking $L$ large enough, guaranteeing $\Re e(U-V)>0$.

\begin{figure}
\vspace*{-2cm}\hspace*{-0.0cm}
\includegraphics[width=140mm,height=100mm]{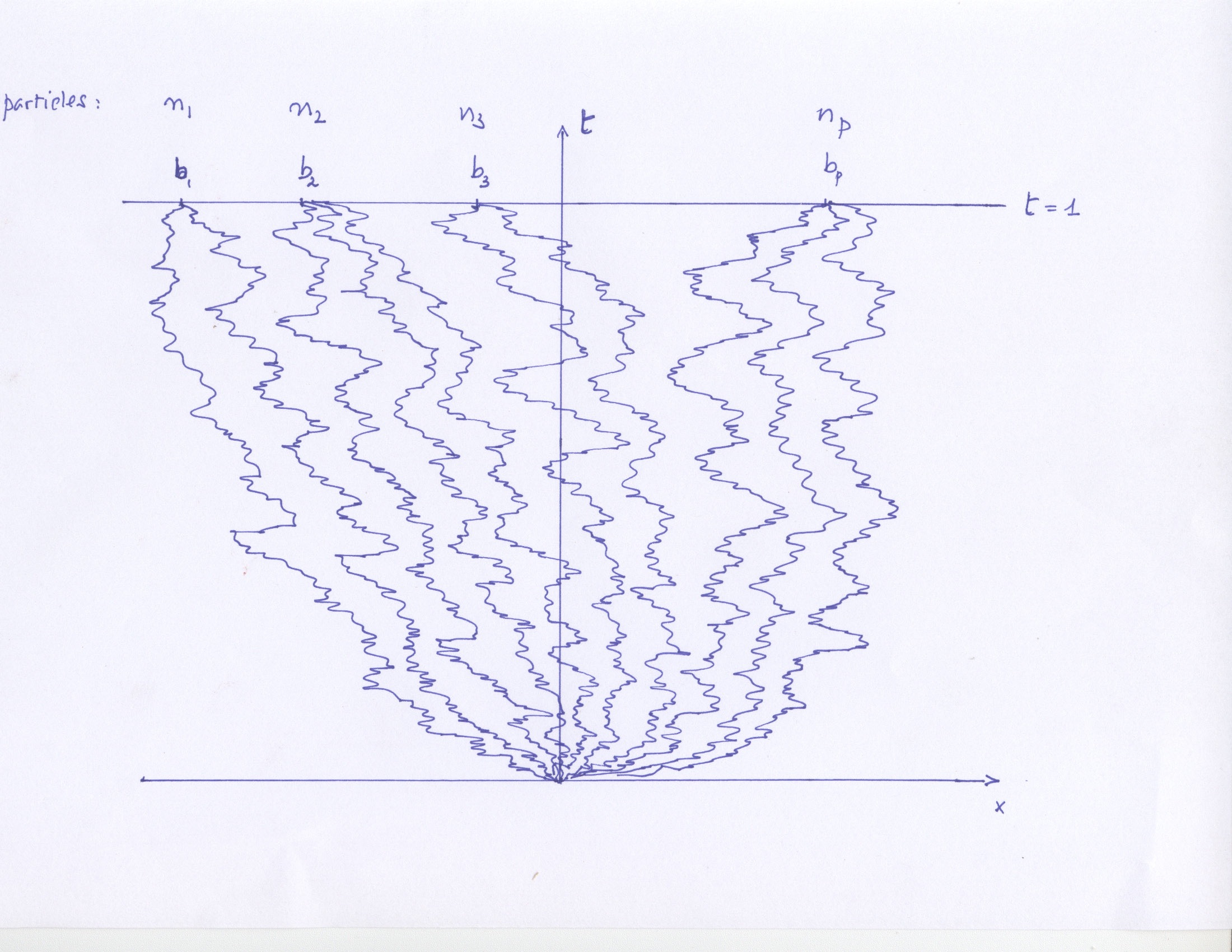}
\vspace*{-2cm}
\caption{Non-intersecting Brownian motions}\label{fig1}
\end{figure}

These non-intersecting Brownian motions $x_1(t)<\ldots<x_N(t)$ describe a diffusion process in a sector
$\{x_1<x_2<\ldots<x_N\}$ of $\BR^N$ and thus satisfy a diffusion equation. When the number $N$ of particles tends
to $\iy$, the transition probability would have to satisfy an ``{\em infinite-dimensional diffusion equation}",
which however would be very difficult to use. The main result of this paper is to show that this transition
probability $\BP^A_n(c,\BE)$ satisfies a non-linear PDE in the boundary points of $E_1,\ldots,E_m$, the target
points $b_1,\ldots,b_p$, and the couplings $c_{1},\dots,c_{m-1}$. It is the determinant of a certain matrix of size
$p+1$; $p$ being the number of target points; so, when the number of particles tends to~$\iy$, the form of this
equation remains the same, which will be exploited in the limit discussed in Theorem \ref{Th: main}. Moreover, this
determinant misses to be a {\em Wronskian} by the last column only.  %
 
The PDE for the transition probability stems largely from integrable theory; this at least is our approach in the
present paper. The integrable theory behind non-intersecting Brownian motions has been developed by us in
\cite{AMVmops}; the latter contains many different ingredients; among them, multi-component KP hierarchies
\cite{UT,MR2150191} and multiple-orthogonal polynomials \cite{AvM-Moser,Van Assche1,BleKui3}.  It is -- in our
opinion~-- an {\em interesting open question} to understand the PDE from a {\em more probabilistic point of view}
and to use more conventional probabilistic tools to derive them. 

\medbreak
 
Throughout the paper, we shall use, without further warning, the following {\bf notation}: (i) The inverse of the
following Jacobi matrix will play an important role: \be \J{}{}\,:=
  \begin{pmatrix}
    -1&&c_1&&&&\\
    &&&\ddots&&0&\\
    c_1&&-1&&\ddots&&\\    &\ddots&&\ddots&&\ddots&\\
    &&\ddots&&-1&&c_{m-1}\\
    &0&&\ddots&&&\\
    &&&&c_{m-1}&&-1
  \end{pmatrix}^{-1}.\label{Jacobi}
\ee
(ii) For any given vector $u=(u_1,\ldots, u_\alpha)$, we denote by
\be
 \pl_u:=\sum_{i=1}^\alpha \frac{\pl}{\pl u_i},~~~ \vr_u:=\sum_{i=1}^\alpha u_i\frac{\pl}{\pl u_i}.
 \label{notation}\ee 
In particular, given any interval or disjoint union of intervals $E=\cup_{i=1}^r [z_{2i-1},z_{2i}]$, we denote by
\bea\label{eq:sum_and_eul}
  \begin{array}{l}
    \pl_E:=\left\{
      \begin{array}{l}
        \mbox{sum of partials in the}\\
        \mbox{boundary points of}~ E
      \end{array}\right\}
      =\sum_{i=1}^{2r} \frac{\pl}{\pl z_i} \\
      \vr_E:=\left\{
  \begin{array}{l}
    \mbox{Euler operator in the}\\ 
    \mbox{boundary points of}~ E
  \end{array}\right\}
   =\sum_{i=1}^{2r} z_i\frac{\pl}{\pl z_i}.
\end{array}\eea
(iii) In view of the Theorem below, given $b=(b_1,\ldots,b_{p-1})$ and subsets~$E_i$, define the linear
differential operators:
\bea
  \pl^{(\ell)}_b&:=&\sum_{i=1}^{p-1}(\kappa_\ell-\delta_{\ell,i})\frac{\pl}{\pl b_i},
  ~~\pl_{b}^{(0)}:=0,~~\mbox{implying}~\sum_{\ell=1}^p\pl_{b}^{(\ell)}=0,\no\\
  \pl_\ell&:=&\pl_{b}^{(\ell)}-\kappa_\ell \sum_{i=1}^m \pl_{_{\!  E_i}}\times 
    \left\{\begin{array}
       {l}{\cal J}_{1i}\mbox{~~~for }\ell=0,\\ {\cal J}_{mi}\mbox{~~for }1\leqs \ell\leqs p,
    \end{array}\right.\no\\
  \vr_b&:=& \sum_1^{p-1}  b_i\frac{\pl}{\pl b_i},\no\\
  \vr_0~&:=&\vr_{_{\! E_1}}-\dt_{1,m}\vr_{_{\! b}}-c_{ 1}\pp{}{c_{ 1}},~~
  \vr_m:=\vr_{_{\! E_m}}-\vr_{_{\! b}}-c_{m- 1}\pp{}{c_{m- 1}}.\label{diff-op}
\eea
For brevity in the statement of the Theorem, set ${}^{\prime}:=\pl_0=\sum_1^m{\cal J}_{1i} \pl_{_{\!E_i}}$.

\begin{theorem}\label{Theo:1.1}
The probability $\BP_n:=\BP^A_n(c,\BE)$, as in (\ref{prob-model}), with the linear constraint (\ref{constraint}) on
the rescaled target points, satisfies a non-linear PDE in the boundary points of the subsets $E_1,\dots, E_m$ and
in the target points $b_1,\dots,b_p$; it is given by the determinant of a $(p+1)\times (p+1)$ matrix, nearly a
Wronskian for the operator ${}^{\prime}:=\pl_0$,
\be
  \det \(
    \begin{array}{ccccccccccc}
      F_1&F_2&F_3&\ldots&F_ p&G_0\\
      F_1'&F_2'&F_3'& \ldots&F_ p'&G_1\\
      F_1'' &F_2''&F_3''& \ldots&F_ p''&G_2\\
      \vdots&\vdots&\vdots&&\vdots&\vdots\\
      F_1^{(p )} &F_2^{(p )}&F_3^{(p )}& \ldots&F_ p^{(p )}&G_p \\
    \end{array}
    \)=0,
\label{PDE}\ee
where the $F_\ell$ and $G_\ell$ are given by
\bea\label{F,G}
  F_{\ell}&=&-\pl_0\pl_\ell\ln\BP_n-n_\ell\J 1m,\no\\
  G_{\ell+1}&:=&\pl_0G_\ell+\sum_{i=1}^p(\pl_0)^\ell F_i\(\pl_0\frac{H_i^{(1)}}{F_i}-\pl_i\frac{H_i^{(2)}}{F_i}\), ~~~~G_0:=0,\no\\
  H_{\ell}^{(1)}&:=&(\kappa_\ell(\delta_{1,m}-\vr_m)\pl_0+2\J1m \pl_b^{(\ell)})\ln\BP_n+C_\ell,\\
  H_{\ell}^{(2)}&:=&(\delta_{1,m}-\vr_0+2\J1m  b_\ell \pl_0)\pl_\ell\ln\BP_n,\no
\eea
with
\be
   C_\ell:=2n_\ell\J 1m\(\J mm  b_\ell -\sum_{i\neq\ell}\frac{n_i}{ b_\ell- b _i }\).
 \ee
 The {\em final condition} (\ref{fc}) translates into an `` {\em initial condition}" near $c_{m-1}\rg 0$ and $b_\ell \rg \iy,$  
 upon using the fact that 
$$
c_{m-1}\simeq \sqrt{1-t_m}, ~~~c_{m-1}b_\ell\simeq O(1).
$$
\end{theorem}

As a special case, we consider the {\em one-time probability} $\BP^{(\beta)}_n(t,\tilde E)$ for
$0<t_1=t<1$. For this case, (\ref{prob-model}) becomes a one-matrix model with external potential
$\BP_n:=\BP^A_n(E)$, thus with no coupling. The expressions for (\ref{F,G}) can be replaced by simpler expressions;
note that the $H_\ell^{(1)}$ in (\ref{1.14}) below are not obtained from the $H_\ell^{(1)}$, as in (\ref{F,G}), by setting $m=1$;
in fact, a further simplification occurs in the equations; also the functions $G_\ell$ are only specializations of
the above $G_\ell$ up to a sign $-(-1)^{\ell}$ and ${}'$ now denotes $\pl_E$ instead of $\pl_0=-\pl_E$. In this
statement, we use the operator $\pl_b^{(\ell)}$ as in (\ref{diff-op}), and we use the following simple operator, in
accord with (\ref{notation}):
$$
  \vr := \vr_E-\vr_b,~~\mbox{with}~ \vr_b=\sum_1^{p-1} b_i\frac{\pl}{\pl b_i}.
 $$

\begin{corollary}\label{Cor:1.2}
%
When $m=1$ (the one-time case), then $\ln\BP_n=\ln\BP^A_n(E)$ satisfies the same non-linear PDE
(\ref{PDE}), but with simpler expressions $F_\ell$ and $H_{ \ell }^{(1)}$ and with ${}'=\pl_{_{E}}$,
%
%
\bea  \label{F,G-one-time}
  F_{\ell} &:=&\left(  \pl^{(\ell)}_b +\kappa_\ell\pl_{_{\! E}}\right)\pl_{_{\! E}} \ln\BP_n+n_{\ell},\no\\ 
  \bar H_{ \ell }^{(1)} &:=&\left(-\kappa_\ell\pl_{_{\! E}}\vr+
    (\kappa_\ell(\vr-1)+2) (\pl^{(\ell)}_b +\kappa_\ell\pl_{_{\! E}} )\right)\ln \BP_n+\bar C_{ \ell },\no\\
  H_{ \ell }^{(2)} &:=& \left(1-\vr +2b_\ell \pl_{_{\! E}}\)\(\pl^{(\ell)}_b +\kappa_\ell\pl_{_{\! E}}\) \ln\BP_n,\label{1.14}\\
  G_{\ell+1}&:=&\pl_{_{\! E}} G_\ell+\sum_{i=1}^p{(\pl_{_{\! E}})^\ell F_i} \(\pl_{_{\! E}}\frac{\bar H_{ i }^{(1)}}{F_{i}}-
              \pl_{b}^{(i)}\frac{H_{ i}^{(2)}}{F_{i}}\right), ~~~~G_0=0, \no\\
  \bar C_\ell&:=&-2n_\ell \left((1-\kappa_\ell)b_\ell+\sum_{j\neq\ell}\frac{n_j}{b_\ell-b_j}\right).\no
\eea
\end{corollary}

In section \ref{sect7}, we shall work out two examples, immediate applications of the equations in Theorem \ref{Theo:1.1} and Corollary \ref{Cor:1.2}. In the first example, we  describe nonintersecting Brownian motions, leaving from $0$ and forced back to $0$. The second example deals with the situation of several target points with the extreme ones being symmetric with regard to the origin. That model will also be used later in Section \ref{section:example}.

\noindent{\bf Pearcey process with inliers}: In section \ref{section:example}, we consider non-intersecting Brownian motions leaving from $0$ and  forced to $p$ target points at time $t=1$, with the only condition that the left-most and right-most target points are symmetric with respect to the origin, with $p-2$ intermediate target points thrown  in totally arbitrarily; it is convenient to rename the target points $\beta_1<\ldots<\beta_p$, as follows:
\be \begin{array}{ccccccccccc}
\tilde a&<& -\tilde c_1&< \ldots <&-\tilde c_{p-2}&<&-\tilde a
\\ \\
~n_{+ }& &  ~n_1           &  \ldots      & n_{p-2}  &  &~n_- 
\end{array}\label{targets}\ee
with the corresponding number of particles forced to those points at time $t=1$. The purpose of this section is to identify the critical process obtained by letting $n:=n_+=n_-\rg \iy$ and by rescaling $\tilde a$ and the $\tilde c_i$ accordingly, while keeping $n_1,\ldots,n_{p-2}$ fixed. We let $\tilde a$ go to $-\iy$ like $-\sqrt{n}$ and $-\tilde a$ to $\iy$ like $\sqrt{n}$. The target points $-\tilde c_1,\ldots,-\tilde c_{p-2}$ of the inliers move to $\iy$ as well, but at a much slower rate, namely like $-u_\ell \left(\frac{n}{2}\right)^{1/4}$. A new process will appear at the point of bifurcation, where the bulk of the particles forced to $-\sqrt{n}$ depart
from those going to $\sqrt{n}$, namely the {\em Pearcey process with inliers}, which generalizes the Pearcey process found by C. Tracy and H. Widom \cite{TW-Pear}. It describes the statistical fluctuations near that point of bifurcation; it will be sensitive to the presence of inliers and will be different in the absence of inliers (Pearcey process). We will compute the kernel governing the transition probabilities and also apply the formulae obtained in Corollary \ref{Cor:1.2} to compute a PDE for the gap probability, which, to our surprise, appears to be an {\em exact $p\times p$ Wronskian}. This is the content of Theorem \ref{Th: main}.




\begin{theorem}\label{Th: main}
Pick times $\tau_1<\ldots<\tau_m$, subsets $E_j\subset \BR$ for $j=1,\ldots,m$ and parameters $u_\ell$ for
$\ell=1,\ldots,p-2$. Consider $2n+\sum_{\ell=1}^{p-2} n_\ell$ non-intersecting Brownian motions, such that

(i) all particles leave from $0$ at time $t=0$,

(ii) $n=n_\pm$ particles are forced to $\pm \sqrt{n}$ at time $t=1$,

(iii) $n_\ell $ paths are forced to points\footnote{Note that those points belong to the interval
  $[-\sqrt{n},\sqrt{n}]$ for large enough $n$.} $-u_\ell \left(\frac{n}{2}\right)^{1/4}$ at time $t=1$ ($1\leqs \ell\leqs p$).


%
\medbreak
\noindent Then the following Brownian motion limit holds for the gap probability, about time $t=1/2$, keeping $n_\ell$ fixed,
\bea
  \lefteqn{\hspace*{-2cm}\lim_{n\rg \iy}\BP\(\bigcap_{j=1}^m\left\{\mbox{all}~x_i
    \Bigl(\frac{1}{2}+\frac{\tau_j}{4\sqrt{2n}}\Bigr)\in\frac{E_j^c}{4(n/2)^{1/4}} \right\}\)}\no\\
  &=&\BP^{\PR~(u_1,\ldots,u_{p-2})}\left( \bigcap_{j=1}^m\left\{{\cal P}(\tau_j)\cap E_j=\emptyset\right\}\right)\no\\
  &=&\det\left(\Id-\left(\raisebox{1mm}{$\chi$}{}_{_{E_i}}K^{\PR}_{\tau_i\tau_j}\raisebox{1mm}{$\chi$}{}_{_{E_j}}
\right)_{1\leqs i,j\leqs m}\right), \label{Pearcey-scaling}
\eea
where this probability is given by the Fredholm determinant of the Pearcey matrix kernel with inliers, which is a
rational perturbation of the customary Pearcey kernel\,\footnote{{\bf X} stands for the contour
 $$ \nwarrow ~  \swarrow $$

\vspace{-.9cm}

$$ 0
$$

\vspace{-.9cm}

$$  \nearrow~ \searrow   $$}, namely
\bea \lefteqn{
 K^{\PR}_{s,t}(X,Y;~u_1,\ldots,u_{p-2})
 }\no\\
 &=&
 -\frac{1
       }{4\pi^2  }\int_{\textbf X} dV
\int^{i\iy}_{-i\iy}dU ~\frac{1}{U-V}
\frac{e^{-\frac{U^4}{4}+\frac{tU^2}{2}- UY}}
 {e^{-\frac{V^4}{4}+\frac{sV^2}{2}-VX}}
\prod_{\ell=1}^{p-2}\left(\frac{U+u_\ell}{V+u_\ell}\right)^{n_\ell} \no\\ \no \\
&&
 -\left\{ \begin{array}{l} 0~~~~~~~~~~~~~~~~~~~~
 ~\mbox{~~for~} t-s\leqs 0\\
\frac{1}{\sqrt{2\pi (t-s)}} e^{-{\frac { \left( { X}-{
Y} \right) ^{2}}{2 ({ t}-
 {
s})}}}
~~\mbox{~~for~} t-s> 0.
  \end{array}\right.
  \label{Pearcey-kernel}\eea
The log of the gap probability $(\BE=E_1\times\dots\times E_m)$
$$
 \BQ(\tau_1,\ldots,\tau_m; u_1,\ldots,u_{p-2};\BE)
 :=\ln \BP^{\PR~(u_1,\ldots,u_{p-2})}\left( \bigcap_{j=1}^m\left\{{\cal P}(\tau_j)\cap
E_j=\emptyset\right\}\right)
$$
satisfies a partial differential equation, which is a $p\times p$ Wronskian with respect to the operator
$\pl_{_{\BE}}=\sum_{i=1}^m \pl_{_{E_i}}$:
\be\label{Wronskian_2}
 {\cal W}_p\left[
     \pl_{_{\!\BE}}^2 \pl_\tau \BQ,~ \pl_{_{\!\BE}}^2 \frac{\pl\BQ}{\pl u_1},\ldots,
     \pl_{_{\!\BE}}^2 \frac{\pl \BQ}{\pl u_{p-2}},
      ~\BX
 \right]_{\pl_{_{\BE}}}=0,
 \ee
where\footnote{Remember for $u=(u_1,\ldots,u_{p-2})$ and $\tau=(\tau_1,\ldots,\tau_m)$, one has
  $\pl_u=\sum_1^{p-2}\frac{\pl}{\pl u_i}$, $\vr_u=\sum_1^{p-2}u_i\frac{\pl}{\pl u_i}$, $\pl_\tau=\sum_1^{m}\frac{\pl}{\pl \tau_i}$.
   One also needs $\vr_\BE :=\sum_{1}^m\vr_{E_i}$ and the mixed time-space derivative $\tilde
  \pl_{_{\!\BE}}:=\sum_{1}^m\tau_i\pl_{_{\!E_i}}$.}  
\be\BX:=
(\vr_{_{\!\BE}}-\vr_u+2\vr_{\tau}-2)\pl_{_{\!\BE}}^2\BQ+4\pl_u\pl_{_{\!\BE}} \pl_\tau\BQ+8\pl^3_{\tau}\BQ
-4\tilde \pl_{_{\!\BE}}\pl_{_{\!\BE}}\pl_\tau\BQ+4\left\{\pl_{_{\!\BE}}\pl_\tau\BQ 
,\pl_{_{\!\BE}}^2\BQ\right\} _{\pl_{_{\BE}}}.
\label{X}\ee
For one-time ($m=1$), the expression $\BX$ reads as follows:
\be\BX:=
 (\vr_{_{\!E}}-\vr_u-2\tau\pp{}{\tau}-2)\pl_{_{\!E}}^2\BQ+4\pl_u\pl_{_{\!E}} \frac{\pl\BQ}{\pl
 \tau}+8\frac{\pl^3\BQ}{\pl
 \tau^3} +4\left\{\pl_{_{\!E}} \frac{\pl\BQ}{\pl \tau}
 ,\pl_{_{\!E}}^2\BQ\right\}
 _{\pl_{_{\!E}}}.
\label{X1}\ee
\end{theorem}

{\bf Remark}: The term $\vr_u \pl_{_{\!E}}^2\BQ$ could be omitted in the definition of $\BX$, since it is a linear combination of $(p-1)$ columns in the matrix (\ref{Wronskian_2}) (from the second $\times u_1$ to the $(p-1)$st column $\times u_{p-2}$). We nevertheless keep this term in the expression, in view of Conjecture \ref{conjecture}.

\

In the absence of inliers, one obtains, in particular, the PDE for the transition probability of the Pearcey
process: it is a $2\times 2$ Wronskian with $\BX$ as in (\ref{X}) and (\ref{X1}), but without the $u$-partials. In
\cite{AOvM}, it is shown that the transition probability of the Pearcey process satisfies the simpler equation
$\BX=0$.

\begin{figure}
\vspace*{-3cm}\hspace*{-1.0cm}
\includegraphics[width=140mm,height=110mm]{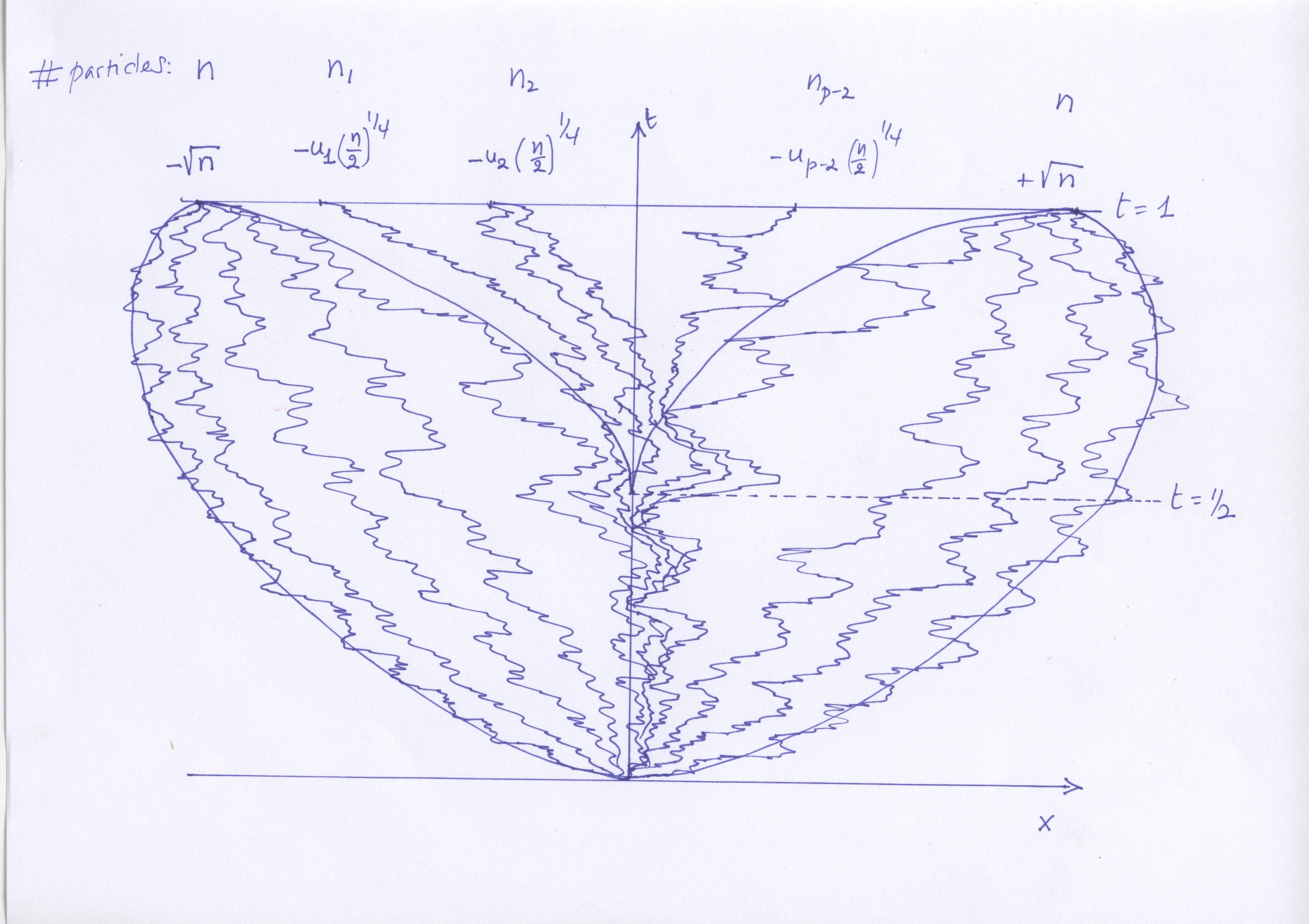}
\vspace*{-2cm}
\caption{Pearcey process with inliers}\label{fig2}
\end{figure}

\begin{corollary}\label{cor:no inl}\cite{AOvM}
In the absence of inliers ($p=2$), 
$$
 \BQ(\tau_1,\ldots,\tau_m;\BE) :=\ln \BP^{\PR }\left(\bigcap_{j=1}^m\{{\cal P}(\tau_j)\cap E_j=\emptyset\}\right)
$$
satisfies
$$
 (\vr_{_{\BE}}+2\vr_{\tau}-2)\pl_{_{\!\BE}}^2\BQ+8\pl^3_{
 \tau}\BQ -4\tilde \pl_{_\BE}\pl_{_\BE}\pl_\tau\BQ+4\left\{\pl_{_{\!\BE}} \pl_{_{\!\tau}}\BQ
 ,\pl_{_{\!\BE}}^2\BQ\right\}
 _{\pl_{_{\BE}}}=0,
$$
and for the one-time case ($m=1$),
$$
 (\vr_{_{\!E}}-2\tau\pp{}{\tau}-2)\pl_{_{\!E}}^2\BQ+8\frac{\pl^3\BQ}{\pl
 \tau^3} +4\left\{\pl_{_{\!E}} \frac{\pl\BQ}{\pl \tau}
 ,\pl_{_{\!E}}^2\BQ\right\}
 _{\pl_{_{\!E}}}=0.
$$
\end{corollary}

We now formulate a conjecture, stating that, even with inliers, the equation for the transition probability reads $\BX=0$, where $\BX$ is given by (\ref{X}) and (\ref{X1}):

\begin{conjecture} \label{conjecture} 
Even with inliers ($p>2$), we conjecture that the function
$$
 \BQ(\tau_1,\ldots,\tau_m; u_1,\ldots,u_{p-2};\BE)
 :=\ln \BP^{\PR~(u_1,\ldots,u_{p-2})}\left( \bigcap_{j=1}^m\left\{{\cal P}(\tau_j)\cap
E_j=\emptyset\right\}\right)
$$
satisfies
\be\BX =
(\vr_{_{\!\BE}}-\vr_u+2\vr_{\tau}-2)\pl_{_{\!\BE}}^2\BQ+4\pl_u\pl_{_{\!\BE}} \pl_\tau\BQ+8\pl^3_{\tau}\BQ
-4\tilde \pl_{_{\!\BE}}\pl_{_{\!\BE}}\pl_\tau\BQ+4\left\{\pl_{_{\!\BE}}\pl_\tau\BQ 
,\pl_{_{\!\BE}}^2\BQ\right\} _{\pl_{_{\BE}}}=0,
\label{Xu}\ee
and for the one-time case ($m=1$),
\be\BX =
 (\vr_{_{\!E}}-\vr_u-2\tau\pp{}{\tau}-2)\pl_{_{\!E}}^2\BQ+4\pl_u\pl_{_{\!E}} \frac{\pl\BQ}{\pl
 \tau}+8\frac{\pl^3\BQ}{\pl
 \tau^3} +4\left\{\pl_{_{\!E}} \frac{\pl\BQ}{\pl \tau}
 ,\pl_{_{\!E}}^2\BQ\right\}
 _{\pl_{_{\!E}}}=0.
\label{Xu1}\ee

\end{conjecture}

\

The PDE's play a {\em prominent role} in obtaining certain approximations which would be very hard to obtain without that technology. An example will be given here, without proof, for the Pearcey process without inliers. At the point of bifurcation, mentioned above, there appears a cusp in the Pearcey scale $\xi= \pm\frac{2}{27}(3\tau)^{3/2}$, such that, roughly speaking, most Pearcey process paths stay completely to the left or to the right of this cusp. Upon comparing the Pearcey process with, say, the right branch of the cusp in the new (crude) space-scale $(3\tau)^{1/6}$, and letting two different times $\tau_1$ and $\tau_2$ tend to $\iy$ in a very specific way, one is led to the so-called Airy process ${\cal A}(t)$. The exact approximation is given in the Theorem below taken from \cite{ACvM}:

 \begin{theorem}\label{Th: Airy by Pearcey}
  Let $\tau_1,\tau_2 \to \iy$, such that 
  $$
  \frac{\tau_2-\tau_1}{2(t_2-t_1)}=(3\tau_1)^{1/3}
  +\frac{t_2-t_1}{(3\tau_1)^{1/3}}+\frac{2t_1t_2}{3\tau_1}+O\bigl(\frac{1}{\tau_1^{5/3}}\bigr)
  ;$$
 this specifies two new times $t_1$ and $t_2$. The following approximation, far out along the cusp, of the Pearcey process by the Airy process holds:
 \bean\lefteqn{ \BP\left(\bigcap_{i=1}^2\left\{\frac{\PR(\tau_i)-\frac{2}{27}(3\tau_i)^{3/2}}{(3\tau_i)^{1/6}}\cap\left(-E _i
\right)=\emptyset   \right\} \right)}\\
&=&
\BP\left( \bigcap_{i=1}^2\left\{{\cal A}(t_i)\cap (-  E_i)=\emptyset 
\right\}\right) \left(1
 +O\bigl(\frac{1}{\tau_1^{4/3}}\bigr)\right)
  .\eean
\end{theorem}
 {\em Remark}: The $O\bigl( {\tau_1^{-4/3}}\bigr)$-approximation, obtained via the PDE is much better than any rough estimate one might predict. Also one expects that, in this precise limit, the Pearcey process with inliers 
   tends to the Airy process with outliers; see \cite{AvMD08}.
%


\section{Non-intersecting Brownian motions and a chain of Coupled Random Matrices}
Setting
$$
  \tau_k:=t_{k+1}-t_k
\mbox{~~and~~}
\frac{1}{\sg_k}:=\frac{1}{t_k-t_{k-1}}+\frac{1}{t_{k+1}-t_k},\quad\mbox{for~}1\leqs k\leqs m,
$$
%
%
and taking in (\ref{tr-pr}) the limit $\gamma_i\to0$, for $i=1,\dots,N$, leads to 
\bea
\lefteqn{\BP \left(
 \displaystyle \bigcap_{k=1}^m\left\{ \mbox{all}~ x_i(t_k)\in \tilde E_k\right\}
  \begin{tabular}{c|l}
    &\ $x_j(0)=0,~~x_j(1)=\delta_j$,  \\
    &\ \mbox{for}~$j=1,\dots,N$
  \end{tabular}\right)\no}
  \\
  &=&\frac1{Z_n'}\int_{\tilde\BE^N}\Delta_N(u_1)\,
      \prod_{k=1}^m\left[
       \det\left(e^{\frac{2u_{k;i}u_{k+1;j}}{\tau_k}}\right)_{1\leqs i,j\leqs N}
      \prod_{1\leqs i\leqs N}
      e^{-  \frac{u_{k;i}^2}{\sigma_k}} du_{k;i}\right]
      ,~~~~
\eea
where $\Delta_N(u_1)$ stands for the Vandermonde determinant in the variables $u_1=(u_{1;1},\dots,u_{1;N})$.
Notice that each of the sets of variables $u_1,\dots,u_{m}$ appears in exactly two of the determinants in the above
integrand and that the other factors are insensitive to a permutation, for fixed $k$ with $1\leqs k\leqs m$, of the
variables $u_k=u_{k;1},\dots,u_{k;N}$. Therefore, taking the limit $u_{m+1;i}=\delta_i\to \beta_j$, for
$i=1,\dots,N$, with $n_\ell$ of the $\delta_i$ going to $\beta_\ell$, namely $u_{m+1;1},\dots,u_{m+1;n_1}\to
\beta_1$, and so on, making $m$ synchronized changes of variables, and using the symmetry of the integration ranges
vis-\`a-vis these variables $u_{k;1},\dots,u_{k;N}$,
%
%
%
\bea \lefteqn{\BP \left(
  \begin{tabular}{c|l}
     &\ $x_j(0)=0$, ($j=1,\dots,N$),\\
     &\ $x_1(1)=\cdots=x_{n_1}(1)=\beta_1$,\\
      $\displaystyle \bigcap_{k=1}^m\left\{ \mbox{all}~ x_i(t_k)\in \tilde E_k\right\} $ 
     &$\qquad\qquad$ $\vdots$\\
     &$x_{N-n_p+1}(1)=\cdots=x_N(1)=\beta_p$
  \end{tabular}\right)}\no\\
  &=&\frac1{Z_n''}\int_{\tilde\BE^N}\Delta_N(u_1)\prod_{\ell=1}^p\left(\Delta_{n_\ell}(\u{\ell}m)
      \prod_{i=1}^{n_\ell}e^{-\sum\limits_{k=1}^{m}
      \frac{{\u\ell{k;i}}^2}{\sigma_k}+\sum\limits_{k=1}^{m-1}\frac{2\u\ell{k;i}\u\ell{k+1;i}}{\tau_k}+
      \frac{2\beta_\ell\u\ell{m;i}}{\tau_m}}\right)\prod_\prodindex\,du_{k;i},\no\\
  &=&\frac1{Z_n'''}\int_{\BE^N}\Delta_N(v_1)\prod_{\ell=1}^p\left(\Delta_{n_\ell}(\v\ell m)
      \prod_{i=1}^{n_\ell}e^{{-\sum\limits_{k=1}^{m}\frac12{{\v\ell{k;i}}^2}+ 
      \sum\limits_{k=1}^{m-1}c_k\v\ell{k;i}\v\ell{k+1;i}}+{b_\ell\v\ell{m;i}}}\right)\prod_\prodindex\,dv_{k;i},\no\\
 &=:&\frac1{Z_n'''}\int_{\BE^N}I_n(v)\prod_{k=1}^m\,dv_k,\label{eq:right_disjoint}\\
  &=&\frac1{\tilde Z_n}\int_{\spec(M_k)\in E_k}e^{-\frac12\tr
       \left(\sum\limits_{k=1}^mM_k^2-2\sum\limits_{k=1}^{m-1}c_kM_kM_{k+1}-2AM_m\right)}\prod_{k=1}^m dM_k,
       \label{coupled matrices}
\eea
where the diagonal matrix $A,~c_k,~\tilde b_\ell$ and $\tilde E_k$ were defined in (\ref{change of variables}) or
alternatively expressed below in terms of the $\sigma_k$'s and $\tau_k$'s. The last integration is taken over Hermitian matrices, with $\spec(M_k)\in E_k$. Also the change of integration variables
$\u\ell{k;i}\mapsto \v\ell{k;i}$ above is given by 
\bean
  \v\ell{k;i}=\sqrt{\frac2{\sigma_k}}\u\ell{k;i},
  \ c_k=\frac{\sqrt{\sigma_k\sigma_{k+1}}}{\tau_k},
  \ b_\ell=\frac{\sqrt{2\sigma_m}}{\tau_m}\beta_\ell,
  \ E_k=\sqrt{\frac2{\sigma_k}}\tilde E_k.
\eean
For $k=1,\dots,m$ and for $\ell=1,\dots,p$, the vector $\u\ell k=(\u\ell{k;1},\dots,\u\ell{k;n_\ell})$ is
defined by
\bean
  (u_{k;1},\dots,u_{k;N})&=&(\u1{k;1},\dots,\u1{k;n_1},\u2{k;1},\dots,\u2{k;n_2},\dots,\u p{k;1},\dots,\u p{k;n_p}).
\eean
%
%
%
Concerning the Jacobi matrix (\ref{Jacobi}), one needs the following formulas for derivatives of $\J{}{}$; they can be shown by recurrence:
\be
  c_1\pp{}{c_1}\J mm=-2\J m1^2,\qquad c_{m-1}\pp{}{c_{m-1}}\J m1=-\J m1(2\J mm+1).\label{eq:J_ders}
\ee

\section{Integrable deformations}
In this section, we introduce a time deformation $\tilde I_n(v)$ of the integrand $I_n(v)$, introduced in
(\ref{coupled matrices}). 
The deformation is chosen such that the resulting integral is on the one hand a solution to the
multi-component KP hierarchy (see \cite{AMVmops} and Proposition \ref{lemma 2.3} below) and satisfies on the other
hand a set of Virasoro constraints. We will impose on the rescaled target points $b_1,\dots, b_p$, which we
henceforth denote by $\b11,\dots,\b p1,$ a non-trivial linear constraint
\be\label{eq:b_constraint}
  \sum_{\ell=1}^p\kappa_\ell\b\ell1=0.
\ee
Without loss of generality, we may assume (upon reordering) that $\kappa_p\neq0$ and impose if
$\sum_1^p\kappa_\ell\neq0$ that $\sum_{\ell=1}^p \kappa_\ell=1$; also define $\kappa_0:=-1$. Thus, the non-deformed
integral which we will consider is
\be
   \int_{\BE^N} \left. I_n(v)\right|_\cons\prod_{k=1}^m\,dv_k.
\ee
The integrand $I_n(v)$ will be deformed by four sets of parameters: ({\bf i})  A first set, denoted by $\b12,\dots,\b p2$, deforms the parameters
$\b\ell1$.  They are subjected to the same constraint
(\ref{eq:b_constraint}) as the parameters $\b\ell1$, namely\footnote{The combination of the two constraints
(\ref{eq:b_constraint}) and (\ref{eq:b2_constraint}) will in the formulas below be denoted by
$\sum_{1}^{p}\kappa_\ell\b \ell {1,2}=0.$}
\be\label{eq:b2_constraint}
  \sum_{\ell=1}^p\kappa_\ell\b\ell2=0.
\ee
({\bf ii}) A second set of deformations consists of parameters corresponding to the KP time variables;
they are denoted by $\s0r$ ($r\in \BZ_{>0}$) for the parameters  going with the starting point $0$ of the
Brownian motion and $\s\ell r$ ($1\leqs\ell\leqs p$ and $r\in \BZ_{>0}$) for the parameters going with the
$\ell$-th end point of the Brownian motion. ({\bf iii}) There is furthermore a set of parameters $\gamma_r^{(k)}$ ($2\leqs
k\leqs m-1$ and $r\in \BZ_{>0}$) going with the intermediate times $t_2,\dots,t_{m-1}$ and ({\bf iv}) a set of parameters
$c^{(k)}_{r,q}$ ($k=1,\dots,m-1$ and\footnote{The inequality $(r,q)>(1,1)$ means by definition that
$r\geqs1,\,q\geqs1$ and $(r,q)\neq (1,1)$.} $(r,q)>(1,1)$), going with consecutive times $t_k,t_{k+1}$.

For $n=(n_1,\dots,n_p)$ and $\BE=E_1\times E_2\times\cdots\times E_m$, where each $E_k$ is the union of a finite
number of intervals in $\BR$, define
\be\label{def:tau_n}
  \tau_n(\BE)
  :=\int_{\BE^N}\left. \tilde I_n(v)\right|_\cons\prod_{k=1}^m\,dv_k,
\ee
where
\bean
  \tilde I_n(v)&=&I_n(v)\times\\
  &&\prod_{\ell=1}^p\prod_{i=1}^{n_\ell}
                   e^{\b\ell2\v\ell{m;i}^2+\sum\limits_{r\geqs1}(\s0 r\v\ell{1;i}^r-\s\ell
              r\v\ell{m;i}^r)+\sum\limits_{k=1}^{m-1}\sum\limits_{\rq}c^{(k)}_{rq}\v\ell{k;i}^r\v\ell{k+1;i}^{q}+
              \sum\limits_{k=2}^{m-1}\sum\limits_{r\geqs 1}\gamma^{(k)}_r\v\ell{k;i}^r},\no\\ \mbox{with}&&\no\\
  I_n(v)&=&\frac{\Delta_N(v_1)}{\prod_{\ell=1}^pn_\ell!}
      \prod_{\ell=1}^p\left(\Delta_{n_\ell}(\v\ell m)
      \prod_{i=1}^{n_\ell}e^{\sum\limits_{k=1}^{m-1}{c_k\v\ell{k;i}\v\ell{k+1;i}}-\frac12\sum\limits_{k=1}^{m}
           {{\v\ell{k;i}}^2}+{\b\ell1\v\ell{m;i}}}\right).\no
\eean
%
%
%
%
We denote by $\LR$ the locus corresponding to setting all deformation parameters equal to zero, so that
$\left.\tilde I_n\right|_\LR=I_n$, 
\be \LR=\left\{
  \begin{array}{l}
    \s0r,\dots,\s pr=0,\quad r\in\BZ_{>0},\\
    \b12,\dots,\b p2=0,\\
    \gamma^{(2)}_r,\dots,\gamma^{(m-1)}_r=0,\quad r\in\BZ_{>0},\\
    c^{(1)}_{rq},\dots, c^{(m-1)}_{rq}=0,\quad\rq
  \end{array}
  \right\}.\label{eq:LR}
\ee
%
%
%
We list a number of operator identities, valid when acting on $\tau_n(\BE)$,
\bea
  \ppb{\ell}h&=&-\pps{\ell}h+\frac{\kappa_\ell}{\kappa_p}\pps{p}h,~~1\leqs\ell\leqs p-1,\,h=1,2,\label{time-identities1}
\\
  \sum_{\ell=1}^p \b{\ell}j\pps\ell h&=&- \sum_{\ell=1}^{p-1}\b\ell j\ppb \ell h,~~h,j\in\{1,2\},\label{time-identities2}
\\
  \pps\ell h&=&-(1-\delta_{\ell,p})\ppb\ell h +\kappa_\ell\(\sum_{i=1}^p\pps ih+\sum_{i=1}^{p-1}\ppb i h\) \no\\
             &=&\p\ell{b_h}+\kappa_\ell\sum_{i=1}^p\pps ih,\label{time-identities4}~~~~~~~~~~h=1,2,\ 1\leqs\ell\leqs p,
\eea
where  for $h=1,2$ and $1\leqs\ell\leqs p$ we define
\bea\label{def:pbk}
  \p{\ell}{b_h}:= -(1-\delta_{\ell,p})\ppb \ell{h}+\kappa_\ell\sum_{i=1}^{p-1}\ppb
  ih=\sum_{i=1}^{p-1}(\kappa_\ell-\delta_{\ell,i})\ppb ih,
\eea
implying
\be
  \sum_{\ell=1}^p  \p{\ell}{b_h}=0.\label{eq:sum=0}
\ee
Using $\sum_{\ell=1}^p \kappa_\ell\b\ell h=0$, one first establishes identity (\ref{time-identities1}) and then
(\ref{time-identities2}), while the first equality in (\ref{time-identities4}) is obtained by computing
$\sum_{i=1}^{p-1}\ppb ih$ from (\ref{time-identities1}) and by using $\sum_{\ell=1}^p\kappa_\ell=1$ and the
identity (\ref{time-identities1}). 

From section 7.3 in \cite{AMVmops}, it follows that $\tau_n(\BE)$ can be written as
\be
  \tau_n(\BE)=\det
  \left(
  \begin{array}{c}
    \left(\langle y^i\varphi_1(y)\mid x^j\psi(x)\rangle\right)_{
      \renewcommand{\arraystretch}{0.6}
        \tiny{
          \begin{array}{c} 
            0\leqs i<n_1\\
            0\leqs j<N
          \end{array}}}\\
    \vdots\\
    \left(\langle y^i\varphi_p(y)\mid x^j\psi(x)\rangle\right)_{
      \renewcommand{\arraystretch}{0.6}
        \tiny{
          \begin{array}{c} 
            0\leqs i<n_p\\
            0\leqs j<N
          \end{array}}}
    \end{array}
    \right),\label{eq:tau_det}
\ee
where 
\bean
  \psi(x)&:=&\exp{\(-\frac12x^2+\sum\limits_{r\geqs1}\s0r x^r\)},\\
  \varphi_\ell(y)&:=&\exp{\(-\frac12y^2+\b\ell1y+\b\ell2y^2-\sum\limits_{r\geqs1}\s\ell r y^r\)},
\eean
for $\ell=1,\dots,p$, and where the inner product $\langle\cdot\mid\cdot\rangle$ is defined by \be \langle f(y)\mid
g(x)\rangle :=\iint_{E_1\times E_m}f(y)g(x){\mu(x,y)}\,dx\,dy,\no \ee
with 
\bea
  \mu(x,y)&:=&\int_{\prod_{k=2}^{m-1}E_k}
    \exp\sum_{k=2}^{m-1}\left(-\frac12w_{k}^2+\sum\limits_{r\geqs 1}\gamma^{(k)}_rw_{k}^r\right)\times\no\\
    &&\qquad\exp\sum_{k=1}^{m-1}\left(c_k w_kw_{k+1}+\sum_\rq c^{(k)}_{rq}w_{k}^rw_{k+1}^{q}\right)\prod_{k=2}^{m-1}dw_{k},\no 
\eea
$w_1:=x$ and $w_{m}:=y$. For $m=2$ the latter formula for $\mu$ should be interpreted as $\mu(x,y):=1$, while
$\mu(x,y):=\delta(x-y)e^{x^2/2}$ (the delta distribution) in the case of $m=1$.

\smallskip

The above representation (\ref{eq:tau_det}) of $\tau_n$ implies, in view of \cite[Prop.\ 6.2]{AMVmops}, that
$\tau_n$ is a tau function of the $p+1$ component KP hierarchy, in particular we have the following Proposition.
\begin{proposition} \label{lemma 2.3}
  The function $\tau_n=\tau_n(\BE)$, as in (\ref{def:tau_n}), satisfies for $1 \leqs\ell\leqs p$
  \begin{equation}\label{C008}
    \pps01\ln\frac{\tau_{n+e_\ell}}{\tau_{n-e_\ell}}= 
        \frac{\frac{\pl^2}{\pl\s02\pl\s\ell1}\ln\tau_n} {\frac{\pl^2}{\pl\s01\pl\s\ell1}\ln\tau_n},~~~~    
    \pps\ell1\ln\frac{\tau_{n+e_\ell}}{\tau_{n-e_\ell}}= 
        -\frac{\frac{\pl^2}{\pl\s01\pl\s\ell2}\ln\tau_n} {\frac{\pl^2}{\pl\s01\pl\s\ell1}\ln\tau_n},
  \end{equation}
  where $n\pm e_\ell=(n_1,\dots,n_p)\pm e_\ell:=(n_1,\dots,n_{\ell-1},n_\ell\pm1,n_{\ell+1},\dots,n_p)$.
\end{proposition}
Both equations will play an important role in Section \ref{sec:PDE} below.

\section{The Virasoro constraints}
%
%
%
Remembering the definition (\ref{eq:sum_and_eul}) of the operators $\pl_E$ and $\vr_E$ and the definition
(\ref{def:pbk}) of the operators $\p\ell{b_k}$, define for $\ell=1,\dots,p$ the operators:
\bea
  \B01&:=&\sum_{k=1}^m \J 1k\pl_{E_k}-2\J 1m\sum_{i=1}^{p-1}\b i2\ppb i1,\label{for:B_defs_1}\\
  \B\ell1&:=&\p\ell{b_1}-\kappa_\ell\( \sum_{k=1}^m\J mk\pl_{E_k}-2\J mm\sum_{i=1}^{p-1}\b i2\ppb i1\),\label{for:B_defs_2}\\
  \B02&:=&-\vr_{E_1}+c_1\ppc1+\delta_{1,m}\left(\sum_{i=1}^{p-1}\b i1\ppb i1+2\sum_{i=1}^{p-1}\b i2\ppb i2\),\label{for:B_defs_3}\\
  \B\ell2&:=&\p\ell{b_2}-\kappa_\ell\(-\vr_{E_m}+c_{m-1}\ppc{m-1}+\sum_{i=1}^{p-1}\b i1\ppb i1+2\sum_{i=1}^{p-1}\b i2\ppb i2\).
 ~~ \label{for:B_defs_4}
\eea
We show in the following proposition, how the action of these operators on the tau function can be represented by
time derivatives.
\begin{proposition} \label{lemma 2.1}
  The integral $\tau_n(\BE)$, as in (\ref{def:tau_n}), satisfies\footnote{Recall that $\kappa_0=-1$.}, for
  $\ell=0,\dots,p$ and $h=1,2$,
  \be
   \B\ell h\ln\tau_n\!\!=\!\!\(\pps\ell h+\kappa_\ell\Sg\ell h\)\ln\tau_n+\kappa_\ell\Tg\ell h, \label{A,B ident}
  \ee
  where
  \bea
    \Tg\a1&=&\left\{
      \begin{array}{ll}
        \ds-\J11 N\s01-\J1m\sum_{\ell=1}^p n_\ell(\b\ell 1-\s\ell1)&\a=0,\\
        \ds-\J1m N\s01-\J mm\sum_{\ell=1}^p n_\ell(\b\ell1-\s\ell1)\qquad&\a\neq0,\\
      \end{array}\right.\label{eq:Ta1}
    \\
    \Tg\a2&=&\left\{
      \begin{array}{ll}
        \ds\sum_{1\leqs i\leqs j\leqs p}n_in_j & m=1,\\
        N(N+1)/2&\a=0\hbox{ and }m>1,\\
        \ds\sum_{\ell=1}^p n_\ell(n_\ell+1)/2\qquad&\a\neq0\hbox{ and }m>1,
      \end{array}\right.\label{eq:Ta2}
  \eea
and each $\Sg\a h$ is a homogeneous first order differential operator in all deformation parameters, except for the
deformation parameters $\b\ell2$, so that $\Sg\a k\Big|_{\LR}=0$, and moreover, for $k=1,2$ and for
$\ell=1,\dots,p$,
\begin{equation}\label{for_sig_comm} 
  \left[\pps\ell1,\Sg 0h\right]=\delta_{h,2}\delta_{1,m}\pps\ell1,\qquad 
  \left[\pps01,\Sg \ell 2\right]=\delta_{1,m}\pps01.
\end{equation}
\end{proposition}
\begin{proof}
We give a detailed proof for the case of $m=2$ (see remark \ref{rem:lemma_m_big} for the case of $m>2$ and see
remark \ref{rem:lemma_m=1} for the special case of $m=1$). Then $c_{m-1}=c_1,$ which we simply write as $c$. Also,
$\J{}{}$ is the $2\times2$ matrix
\begin{equation*}
  \J{}{}=\begin{pmatrix} -1&c\\ c&-1\end{pmatrix}^\mi=\frac{-1}{1-c^2}
          \begin{pmatrix} 1&c\\ c&1\end{pmatrix}.
\end{equation*}%
In this case, referring to (\ref{def:tau_n}), there are two sets of variables $v_1$ and $v_2$, which we denote by
$x$ and $y$, there are no deformation parameters $\gamma^{(k)}_r$ and there is a single set of deformation
parameters $c^{(1)}_{rq}$, which we will denote by $c_{rq}$. For $E_1,E_2\subset\BR$, and taking into account the
usual constraint $\sum_{\ell=1}^{p}\kappa_\ell\b \ell {1,2}=0$,
\be
  \tau_n(E_1,E_2)
  :=\frac1{\prod_{\ell=1}^p n_\ell !}\iint_{E_1^N\times E_2^N} \tilde I_n(x,y)\,dx\, dy,
\label{tau-one-time}\ee
where
\bea
  \tilde I_n(x,y)&=&\Delta_N(x)
      \prod_{\ell=1}^p\left(\Delta_{n_\ell}(y)
      \prod_{i=1}^{n_\ell}e^{-\frac12{\x\ell i}^2-\frac12{\y\ell i}^2+{c\x\ell{i}\y\ell{i}}+
        {\b\ell1\y\ell{i}}}\right.\times\label{eq:3.8'}
   \\
  &&\qquad\qquad\quad\left.
                   e^{\b\ell2\y\ell{i}^2+\sum\limits_{r\geqs1}(\s0 r\x\ell{i}^r-\s\ell
              r\y\ell{i}^r)+\sum\limits_{\rq}c^{}_{rq}\x\ell{i}^r\y\ell{i}^{q}}\right).\no
\eea

We first compute the action of the operators $\pl_{E_k}$ and $\vr_{E_k}$ on the tau function
(\ref{tau-one-time}).  We start with $\pl_{E_2}$. Using the fundamental theorem of calculus and the fact that
$\sum_{i=1}^N \frac{\pl}{\pl y_i}\Dt_N(y)=0$, we compute from (\ref{eq:3.8'}) that
{\footnotesize 
\bea
 \pl_{E_2}\tau_n \!\!&=&\iint\limits_{E_1^N\times E_2^N}
   \sum_{i=1}^N\pp{\tilde I_n}{y_i}(x,y)
   \,dx\,dy\no\\
  &=&\iint\limits_{E_1^N\times E_2^N}\sum_{\ell=1}^p\sum_{j=1}^{n_\ell}\(-\y\ell j+c\x\ell j+\b\ell1-
          \sum_{k\geqs1}k\s\ell k(\y\ell j)^{k-1}+2\b\ell2\y\ell j\right.\no\\
  &&\qquad\qquad  
   \left.+\sum_\rq q c_{rq}(\x\ell j)^r(\y\ell j)^{q-1}\right) \tilde I_n(x,y)
   \,dx\,dy\no\\
  &=&\iint\limits_{E_1^N\times E_2^N}\(\sum_{\ell=1}^p\pps\ell1+c\pps01+\sum_{\ell=1}^pn_\ell(\b\ell1-\s\ell1)
         +\sum_{\ell=1}^p\sum_{k\geqs2}k\s\ell k\pps\ell{k-1}
          \right.\no\\
  &&\qquad\qquad
 \left.   -2\sum_{\ell=1}^p\b\ell2\pps\ell1 +\sum_{r\geqs2}c_{r1}\pps0r+\!\!\!\!\sum_{\renewcommand{\arraystretch}{1}
        \tiny\begin{array}{c}\rq\\q\geqs2\end{array}} 
         q c_{rq}\ppc{r,q-1}\right) \tilde  I_n(x,y)\,dx\,dy
         \no\\
  &=&\(\sum_{\ell=1}^p\pps\ell1+c\pps01+\sum_{\ell=1}^pn_\ell(\b\ell1-\s\ell1)
         +\sum_{\ell=1}^p\sum_{k\geqs2}k\s\ell k\pps\ell{k-1}+2\sum_{\ell=1}^{p-1}\b\ell2\ppb\ell1\right.\no\\
  &&\qquad\qquad  \left.+\sum_{r\geqs2}c_{r1}\pps0r+\sum_{\renewcommand{\arraystretch}{1}
        \tiny\begin{array}{c}\rq\\q\geqs2\end{array}} q c_{rq}\ppc{r,q-1}\right) \tau_n,\no
\eea}%
where we have used the identity (\ref{time-identities2}), which follows from the constraint
$\cons$, in the last step. The computation for $\pl_{E_1}$ is similar, but
simpler:
{\footnotesize
\bean
  \lefteqn{\pl_{E_1}\tau_n}\\
  &=&\iint\limits_{E_1^N\times E_2^N}\sum_{i=1}^N\pp{\tilde I_n}{x_i}(x,y)\,dx\,dy\no\\
  &=&\iint\limits_{E_1^N\times E_2^N}\sum_{\ell=1}^p\sum_{j=1}^{n_\ell}\(-\x\ell j+c\y\ell j+
     \sum_{k\geqs1}k\s0 k(\x\ell j)^{k-1}+\sum_\rq rc_{rq}(\x\ell j)^{r-1}(\y\ell j)^{q}\)\tilde I_n(x,y)\,dx\,dy\no\\
  &=&\(-\pps01-c\sum_{\ell=1}^p\pps\ell1+N\s01+\sum_{k\geqs2}k\s0k\pps0{k-1}-\sum_{\ell=1}^p\sum_{q\geqs2}c_{1q}\pps\ell q\right.
  \left.+\sum_{\renewcommand{\arraystretch}{1}
        \tiny\begin{array}{c}\rq\\r\geqs2\end{array}} r c_{rq}\ppc{r-1,q}\)\tau_n.\no
\eean}%
For the computation of the action of $\vr_{_{E_1}}$ and $\vr_{_{E_2}}$ on the tau function, note 
\begin{equation*}
  \sum_{i=1}^N\pp{}{x_i}(x_if)=Nf+  \sum_{i=1}^Nx_i\pp{f}{x_i},\qquad
  \sum_{i=1}^Nx_i\pp{}{x_i}\Delta_N(x)=\frac{N(N-1)}2\Delta_N(x),
\end{equation*}%
and so from (\ref{eq:3.8'}), compute using (\ref{time-identities2}) and the constraints $\cons$,
{\footnotesize\bea
  \vr_{E_2}\tau_n&=&\iint\limits_{E_1^N\times E_2^N}
  \sum_{i=1}^N\pp{}{y_i}(y_i  \tilde I_n(x,y))
   \,dx\,dy\no\\
  &=&\iint\limits_{E_1^N\times E_2^N}\left(N+\sum_{\ell=1}^p\frac{n_\ell(n_\ell-1)}2+\right.\no\\
  &&\qquad\sum_{\ell=1}^p\sum_{j=1}^{n_\ell}\({-\y\ell j}^2+c\x\ell j\y\ell j+\b\ell1\y\ell j-
          \sum_{k\geqs1}k\s\ell k(\y\ell j)^{k}+2{\b\ell2\y\ell j}^2\right.\no\\
  &&\qquad\qquad 
   \left.\left.+\sum_\rq q c_{rq}(\x\ell j)^r(\y\ell j)^q\right)\right) \tilde I_n(x,y)
   \,dx\,dy\no\\
  &=&\iint\limits_{E_1^N\times E_2^N}\(\sum_{\ell=1}^p\frac{n_\ell(n_\ell+1)}2+\sum_{\ell=1}^p\pps\ell2+c\pp{}c-
      \sum_{\ell=1}^p\b\ell1\pps\ell1+\sum_{\ell=1}^p\sum_{k\geqs1}k\s\ell k\pps\ell{k}\right.\no\\
  &&\qquad\qquad  
   \left.-2\sum_{\ell=1}^p\b\ell2\pps\ell2+\sum_\rq q c_{rq}
          \ppc{rq}\right) \tilde I_n(x,y)
           \,dx\,dy\no\\
 \eea \bea &=&\(\sum_{\ell=1}^p\frac{n_\ell(n_\ell+1)}2+\sum_{\ell=1}^p\pps\ell2+c\pp{}c+
      \sum_{\ell=1}^{p-1}\b\ell1\ppb\ell1+\sum_{\ell=1}^p\sum_{k\geqs1}k\s\ell k\pps\ell{k}\right.\no\\
  &&\qquad\qquad  \left.+2\sum_{\ell=1}^{p-1}\b\ell2\ppb\ell2+\sum_\rq q c_{rq}\ppc{rq}\right)\tau_n.\no
\eea}
Similarly,
\bea
 \vr_{E_1}\tau_n &=&\iint\limits_{E_1^N\times E_2^N}\sum_{i=1}^N\pp{}{x_i}(x_i  \tilde I_n(x,y))\,dx\,dy\no\\
  &=&\(\frac{N(N+1)}2-\pps02+c\pp{}c+\sum_{k\geqs1}k\s0 k\pps0{k}+\sum_\rq r c_{rq}\ppc{rq}\)\tau_n.\no
\eea
In order to deduce (\ref{A,B ident}) from these formulas it suffices, for $h=1$, to substitute in the first line
the definitions (\ref{for:B_defs_1}), (\ref{for:B_defs_2}) for $\B\ell1$ and in the second line the expressions for
$\pl_{E_1}\tau_n$ and $\pl_{E_2}\tau_n$ in\footnote{Recall that $m=2$ and that $\kappa_0=-1$.}
\bea
  \begin{pmatrix}
    \B01\\ \B\ell1
  \end{pmatrix}
  \tau_n
  &=&\left\{-
  \begin{pmatrix}
    \kappa_0&0\\0&\kappa_\ell
  \end{pmatrix}
  \J{}{}
  \begin{pmatrix}
    \pl_{E_1}\\ \pl_{E_2}-2\sum_{i=1}^{p-1}\b i2\ppb i1
  \end{pmatrix}
  +
  \begin{pmatrix}
    0\\ \p\ell{b_1}
  \end{pmatrix}\right\}\tau_n
  \no\\
  &=&
  \left\{
  \begin{pmatrix}
    \pps01\\ \p\ell{b_1}+\kappa_\ell\sum\limits_{i=1}^p\pps i1
  \end{pmatrix}
  -
  \begin{pmatrix}
    \kappa_0&0\\0&\kappa_\ell
  \end{pmatrix}
\J{}{}
  \begin{pmatrix}
    N\s01\\
    \sum\limits_{i=1}^pn_i(\b i1-\s i1)
  \end{pmatrix}
  +
  \begin{pmatrix}
    \kappa_0\Sg01\\
    \kappa_\ell\Sg\ell1
  \end{pmatrix}
  \right\}\tau_n
  \no\\
  &=&
  \left\{
  \begin{pmatrix}
    \pps01\\ \pps\ell1
  \end{pmatrix}
  +
  \begin{pmatrix}
    \kappa_0(\Tg01+\Sg01)\\    \kappa_\ell(\Tg\ell1+\Sg\ell1)
  \end{pmatrix}
  \right\}\tau_n,
  \no
\eea
where we used (\ref{time-identities4}) (for $k=1$) in the third line, and where we set 
\be
  \begin{pmatrix}
    \Sg01\\
    \Sg\ell1
  \end{pmatrix}
  :=-\J{}{}
    \begin{pmatrix}
      \sum_{k\geqs2}k\s0k\pps0{k-1}-\sum_{i=1}^p\sum_{q\geqs2}c_{1q}\pps\ell q+\sum_{\renewcommand{\arraystretch}{1}
        \tiny\begin{array}{c}\rq\\r\geqs2\end{array}} r c_{rq}\ppc{r-1,q}\\
      \sum_{\ell=1}^p\sum_{k\geqs2}k\s\ell k\pps\ell{k-1}+\sum_{r\geqs2}c_{r1}\pps0r+\sum_{\renewcommand{\arraystretch}{1}
        \tiny\begin{array}{c}\rq\\q\geqs2\end{array}} q c_{rq}\ppc{r,q-1}
    \end{pmatrix}.\label{eq:S0ell}
\ee
Thus we see that $\Sg01$ and $\Sg\ell1$ are homogeneous first order differential operators in the deformation
parameters, and that they are independent of $\s11,\dots,\s p1$, and of $\b12,\dots,\b p2$, leading to the stated
properties of $\Sg01$ and $\Sg\ell1$. For $k=2$, it suffices to substitute the found expressions for $\vr_{E_1}$ and
$\vr_{E_2}$, acting on $\tau_n$, in the definitions (\ref{for:B_defs_3}) and (\ref{for:B_defs_4}) of $\B02$ and
$\B\ell2$, to wit:
\bea
  \B02\tau_n&=&\(-\vr_{E_1}+c\pp{}c\)\tau_n\no\\
      &=&\(\pps02-\frac{N(N+1)}2-\sum_{k\geqs1}k\s0 k\pps0{k}-\sum_\rq r c_{rq}\ppc{rq}\)\tau_n,\no\\
      &=:&\(\pps02+\kappa_0\frac{N(N+1)}2+\kappa_0\Sg02\)\tau_n,\no
\eea
and 
\bea
  \B\ell2\tau_n
  &=&\(\p\ell{b_2}+\kappa_\ell\(\vr_{E_2}-c\pp{}c-\sum_{i=1}^{p-1}\b i1\ppb i1-2\sum_{i=1}^{p-1}\b i2\ppb i2\)\)\tau_n\no\\
  &=&\(\p\ell{b_2}+\kappa_\ell\sum_{i=1}^p\pps i2+\right.\\
  &&
   \left.\kappa_\ell\(\sum_{i=1}^p\frac{n_i(n_i+1)}2+
      \sum_{i=1}^p\sum_{k\geqs1}k\s ik\pps ik+\sum_{\rq}qc_{rq}\pp{}{c_{rq}}\)\)\tau_n\no\\
  &=:&\(\pps\ell2+\kappa_\ell\sum_{i=1}^p\frac{n_i(n_i+1)}2+\kappa_\ell\Sg\ell2\)\tau_n,\no
\eea
where\footnote{Notice that $\Sg\ell2$ is independent of $\ell$ for $1\leqs\ell\leqs p$.}
\bea\label{eq:sigval}
  \Sg02&=&\sum_{k\geqs1}k\s0 k\pps0{k}+\sum_\rq r c_{rq}\ppc{rq},\no\\
  \Sg\ell2&=& \sum_{i=1}^p\sum_{k\geqs1}k\s ik\pps ik+\sum_{\rq}qc_{rq}\pp{}{c_{rq}}.
\eea
\end{proof}

\begin{remark}\label{rem:lemma_m_big}
{\rm For $m>2$ the proof goes along the same line, but it has extra terms, coming from the deformation parameters
$\gamma^{(k)}_r$. As it turns out, 
\be
  \pp{}{\gamma^{(k)}_1}\tau_n=\sum_{i=1}^m\J
  ki\(\pl_{E_i}-\delta_{i,m}\(2\sum_{\ell=1}^{p-1}\b\ell2\ppb\ell1+\sum_{\ell=1}^p n_\ell\b\ell1\)\)\tau_n+O(\LR),
\ee
while $\pps i1\tau_n$ are as before, $\mod O(\LR)$, so the $\pp{}{\gamma^{(k)}_1}\tau_n$ are only needed to solve for
$\pps i1\tau_n$ in terms of the $(\pl_{E_i}-\delta_{i,m}(\star))\tau_n$, but they do not enter into the actual
solution of $\pps i1\tau_n\mod O(\LR)$.}
\end{remark}

\begin{remark}\label{rem:lemma_m=1}
{\rm For $m=1$ (one time) the proof of Proposition \ref{lemma 2.1} is simpler, but a few adjustments are needed. Denoting
the subset $E_1\subset\BR$ by $E$, setting $\kappa_0:=-1$ and $\p0{b_1}:=\p0{b_2}:=0$, the operators $\B\ell1$ and
$\B\ell2$ can for $\ell=0,\dots,p$, be written as
\bea 
  \B\ell1&=&\p\ell{b_1}+\kappa_\ell\(\pl_E-2\sum_{i=1}^{p-1}\b i2\ppb i1\),\no\\
  \B\ell2&=&\p\ell{b_2}+\kappa_\ell\(\vr_E-\sum_{i=1}^{p-1}\b i1\ppb i1-2\sum_{i=1}^{p-1}\b i2\ppb i2\),\label{eq:B_m=1}
\eea 
while $T_k:=T^{(\a)}_k$ and $\Sigma_k:=\Sg\alpha k$ are independent of $\a$ and take the simple form
\bea 
  T_1&=&N\s01+\sum_{\ell=1}^p n_\ell(\b\ell1-\s\ell1),\qquad  T_2=\sum_{1\leqs i\leqs j\leqs  p}n_in_j,\label{eq:Tm=11}\\
  \Sigma_1&=&\sum_{\ell=0}^p\sum_{k\geqs2}k\s\ell k\pps\ell{k-1},\qquad
  \Sigma_2=\sum_{\ell=0}^p\sum_{k\geqs1}k\s\ell k\pps\ell{k}.
\eea }
\end{remark}

\section{Virasoro constraints, restricted to the locus $\LR$}
Restricting the operators $\BB_i$, $T_i$ and $\Sigma_i$ ((\ref{for:B_defs_1}) -- (\ref{eq:Ta2})) to the locus
$\LR$, defined by setting all deformation parameters equal to zero (see (\ref{eq:LR})), yields new operators for
$\ell=0,\dots,p$,
\bea
 \hB\ell1 &:=& \p\ell{b_1}-\kappa_\ell \sum_{i=1}^m \pl_{E_i}\times \left\{\begin{array}
 {l}{\cal J}_{1i}\mbox{~~~for }\ell=0\\ {\cal J}_{mi}\mbox{~~for }1\leqs \ell\leqs p\end{array}\right.\no\\
  \hB02&:=&-\vr_{E_1}+c_1\ppc1+\delta_{1,m}\sum_{i=1}^{p-1}\b i1\ppb i1\no\\
  \hB\ell2&:=&\!\p\ell{b_2}-\kappa_\ell\left(\!-\vr_{E_m}+c_{m-1}\ppc{m-1}+\sum_{i=1}^{p-1}\b i1\ppb i1\!\right),~\mbox{for}~~\ell\geqs 1. 
 \label{eq:hBs}\eea
while all $\Sg\ell{k}$, defined in (\ref{eq:S0ell}) and (\ref{eq:sigval}), restrict to zero, $\hTg\ell2=\Tg\ell2$
for $0\leqs \ell\leqs p$ and
\begin{equation}\label{hT_eqs}
  \hTg01=-\J1m N(b_1),\quad  \hTg\ell1=-\J mm N(b_1),\ \hbox{for }\ 1\leqs \ell\leqs p, \quad
\end{equation}%
where $N(b_1):=\sum_{\ell=1}^pn_\ell\b\ell1$. It leads, on the locus $\LR$, to the identities:

\begin{proposition}\label{lemma 2.2}
 For $\ell=0\dots,p$ and $h=1,2$, the following formulas hold on the locus~$\LR$:
\bea
  \frac{\pl}{\pl s_h^{(\ell)}}\ln \tau_n&=&\hB\ell h\ln\tau_n-\kappa_\ell \hTg\ell h,
        \label{1st der}
\eea
while for second derivatives and $\ell=1,\dots,p,$ also on the locus~$\LR$,
\bea
  \frac{\pl^2}{\pl\s01\pl\s\ell1}\ln\tau_n&=&\hB01\hB\ell1\ln\tau_n+n_\ell\J 1m=:-F_\ell,\no\\
  \frac{\pl^2}{\pl\s02\pl\s\ell1}\ln\tau_n&=&(\hB02+\delta_{1,m})\hB\ell1\ln\tau_n-2 \J1m^2\kappa_\ell N(b_1), \label{2nd der} \\  
  \frac{\pl^2}{\pl\s01\pl\s\ell2}\ln\tau_n&=&(\hB\ell2-\kappa_\ell\delta_{1,m})\hB01\ln\tau_n-2\J 1m(\J mm\kappa_\ell N(b_1)+\p\ell{b_1}\ln\tau_n).\no
\eea
\end{proposition}

\begin{proof}
The first set of identities (\ref{1st der}) follows at once from restricting the identities (\ref{A,B ident}) of
Proposition \ref{lemma 2.1} to the locus $\LR$ and using that $\pps\ell k$ and $\B\ell k$ are first order differential
operators. The identities (\ref{2nd der}) involving second derivatives are shown as follows. Concerning the first
one, observe from Proposition \ref{lemma 2.1} that
\bean 
  \lefteqn{\hB01\hB\ell1\ln\tau_n\Big|_{\LR}=
     \B01\B\ell1\ln\tau_n\Big|_{\LR}=\B01\(\pps\ell1+\kappa_\ell\Sg\ell1\)\ln\tau_n\Big|_{\LR}+
         \kappa_\ell\B01\Tg\ell1\Big|_{\LR}}\\ \\ 
    &=&\(\pps\ell1+\kappa_\ell\Sg\ell1\)\B01\ln\tau_n\Big|_\LR
       =\pps\ell1\B01\ln\tau_n\Big|_{\LR}\\ \\
    &=&\pps\ell1\((\pps01+\kappa_0\Sg0{1})\ln\tau_n+\kappa_0\Tg01\)\Big|_{\LR}=
         \frac{\pl^2}{\pl\s\ell1\pl\s01}\ln\tau_n\Big|_{\LR}-\J1mn_{\ell}, 
\eean
where we used in the last equality the relations $\frac{\pl}{\pl\s\ell1}\Tg01=\J1mn_\ell$ (see (\ref{eq:Ta1})) and
$\left[\pps\ell1,\Sg01\right]=0,$ (see (\ref{for_sig_comm})). This yields the first identity in (\ref{2nd der}). To
prove the third one, we use that
$$
  \sum_{i=1}^{p-1}\p\ell{b_2}(\b i2)\pp{}{\b i1}=\sum_{i=1}^{p-1}(\kappa_\ell-\delta_{\ell,i})\pp{}{\b
  i1}=\p\ell{b_1},
$$
as follows from (\ref{def:pbk}), and 
\bean
  \B\ell2\Tg01\Big|_\LR&=&\kappa_\ell c_{m-1}\pp{\J1m}{c_{m-1}}N(b_1)+\kappa_\ell\J1m\sum_{i=1}^{p-1}\b i1\ppb i1N(b_1)\\
  &=&-\kappa_\ell\J1m(2\J mm+1)N(b_1)+\kappa_\ell\J1m N(b_1)=-2\kappa_\ell\J1m\J mmN(b_1),
\eean
by using (\ref{eq:J_ders}), when $m>1$, and $\B\ell2\Tg01\Big|_\LR=-\kappa_\ell N(b_1)$, by Remark
\ref{rem:lemma_m=1} for $m=1$, so that
$$
  \B\ell2\Tg01\Big|_\LR=-\kappa_\ell N(b_1)(2\J1m\J mm-\delta_{1,m}),
$$
for all $m$.  Using these identities, (\ref{for:B_defs_1}), (\ref{for:B_defs_4}), Proposition \ref{lemma 2.1},
(\ref{for_sig_comm}) and (\ref{1st der}), compute
{\footnotesize \bean 
\lefteqn{\hB\ell2\hB01\ln\tau_n\Big|_{\LR}}\\
    &=&\B\ell2\B01\ln\tau_n\Big|_{\LR}+2\J 1m\sum_{i=1}^{p-1}\p\ell{b_2}(\b i2)\pp{}{\b i1}\ln\tau_n\Big|_{\LR}\\
    &=&\B\ell2\(\pps01+\kappa_0\Sg01\)\ln\tau_n\Big|_{\LR}+\kappa_0\B\ell2\Tg01\Big|_{\LR} +
              2\J 1m\p\ell{b_1}\ln\tau_n\Big|_{\LR}\\ \\ 
    &=&\(\pps01+\kappa_0\Sg01\)\B\ell2\ln\tau_n\Big|_{\LR}+\kappa_\ell N(b_1)(2\J1m\J mm-\delta_{1,m})
              +2\J1m\p\ell{b_1}\ln\tau_n\Big|_{\LR}\\ \\
    &=&\pps01\(\!\!\!\(\pps\ell2+\kappa_\ell\Sg\ell2\)\ln\tau_n+\kappa_\ell\Tg\ell2\)\Big|_{\LR}
        +\kappa_\ell N(b_1)(2\J1m\J mm-\delta_{1,m})+2\J1m\p\ell{b_1}\ln\tau_n\Big|_{\LR}\\ \\ 
    &=&\pp{{}^2}{\s01\pl\s\ell2}\ln\tau_n\Big|_{\LR}+\kappa_\ell\left[\pps01,\Sg\ell2\right]\ln\tau_n\Big|_\LR
        +\kappa_\ell N(b_1)(2\J1m\J mm-\delta_{1,m})+2\J1m\p\ell{b_1}\ln\tau_n\Big|_{\LR}\\ \\ 
    &=&\pp{{}^2}{\s01\pl\s\ell2}\ln\tau_n\Big|_{\LR}+\delta_{1,m}\kappa_\ell
                \(\pps01\ln\tau_n\Big|_{\LR}-N(b_1)\)+2\J1m(\kappa_\ell N(b_1)\J mm +\p\ell{b_1}\ln\tau_n\Big|_{\LR})\\
    &=&\pp{{}^2}{\s01\pl\s\ell2}\ln\tau_n\Big|_{\LR}+
       \delta_{1,m}\kappa_\ell\hB01\ln\tau_n\Big|_{\LR}+2\J1m(\kappa_\ell N(b_1)\J mm +\p\ell{b_1}\ln\tau_n\Big|_{\LR}),
\eean}%
which yields the third relation (\ref{2nd der}). Using
$\left.\B02\Tg\ell1\right|_\LR=N(b_1)(2\J1m^2-\delta_{1,m})$, which follows from (\ref{for:B_defs_3}),
(\ref{eq:Ta1}) and (\ref{eq:J_ders}), the second identity in (\ref{2nd der}) is proven in a similar fashion, using
(\ref{for_sig_comm}) and (\ref{1st der}), namely
\bean 
\hB02\hB\ell1\ln\tau_n\Big|_{\LR}
    &=&\B02\B\ell1\ln\tau_n\Big|_{\LR}\\
    &=&\B02\(\pps\ell1+\kappa_\ell\Sg\ell1\)\ln\tau_n\Big|_{\LR}+\kappa_\ell\B02\Tg\ell1\Big|_{\LR}\\ \\ 
    &=&\(\pps\ell1+\kappa_\ell\Sg\ell1\)\B02\ln\tau_n\Big|_{\LR}+\kappa_\ell N(b_1)(2\J1m^2-\delta_{1,m})\\ \\ 
    &=&\pps\ell1\(\(\pps02+\kappa_0\Sg02\)\ln\tau_n+\kappa_0\Tg02\)\Big|_{\LR}+\kappa_\ell N(b_1)(2\J1m^2-\delta_{1,m})\\ \\ 
    &=&\pp{{}^2}{\s\ell1\pl\s02}\ln\tau_n\Big|_{\LR}-\left[\pps\ell1,\Sg02\right]\ln\tau_n\Big|_{\LR}
           +\kappa_\ell N(b_1)(2\J1m^2-\delta_{1,m})\\
 \eean\bean   &=&\pp{{}^2}{\s\ell1\pl\s02}\ln\tau_n\Big|_{\LR}-\delta_{1,m}\pps\ell1\ln\tau_n\Big|_{\LR}
          +\kappa_\ell N(b_1)(2\J1m^2-\delta_{1,m})\\
    &=&\pp{{}^2}{\s\ell1\pl\s02}\ln\tau_n\Big|_{\LR}+2\kappa_\ell N(b_1)\J1m^2-\delta_{1,m}\hB\ell1\ln\tau_n\Big|_{\LR}.
\eean
\end{proof}

\section{A PDE for the transition probability}\label{sec:PDE}

This section aims at proving Theorem \ref{Theo:7.2}, which leads at once to Theorem~\ref{Theo:1.1}. In order to do
so, we shall need two propositions:
\begin{proposition}\label{lemma 2.4}
  For $1\leqs\ell\leqs p$, the function $X_\ell:=\pl^{(\ell)}_{b_2}\hB01\ln\tau_n\Bigl|_{\LR}$ satisfies the
  equation
  \begin{equation}\label{A2}
    \left\{X_\ell,\,F_\ell \right\}_{\hB01}=
    \left\{H_\ell^{(1)},F_{\ell}\right\}_{\hB01}-\left\{H_{\ell}^{(2)}, F_{\ell}\right\}_{\hB\ell1},
   \end{equation}
where ($n=(n_1,\ldots,n_p)$)
 \bea
    F_{\ell}&=&-\hB01\hB\ell1\ln\BP_n-n_\ell\J 1m,\no\\
    H_{\ell}^{(1)}&:=&(\kappa_\ell(\delta_{1,m}-\vr_m)\hB01+2\J1m\p\ell{b_1})\ln\BP_n+C_\ell,\no\\
    H_{\ell}^{(2)}&:=&(\hB02+\delta_{1,m}+2\J1m\b\ell1\hB01)\hB\ell1\ln\BP_n,\no\\
    \vr_m&=&\vr_{E_m}-c_{m-1}\pp{}{c_{m-1}}-\sum_{\ell=1}^{p-1}\b\ell1\ppb\ell1,\label{E}
     \\
    C_\ell&:=&2n_\ell\J 1m\(\J mm \b\ell1-\sum_{i\neq\ell}\frac{n_i}{\b\ell1-\b i1}\).
 \label{7.4.5}  \eea%
\end{proposition}

\begin{proof}
From (\ref{1st der}) and (\ref{hT_eqs}), one finds, along $\LR$, for $\ell=1,\dots,p$,
\bea 
  \pps01\ln\frac{\tau_{n+e_{\ell}}}{\tau_{n-e_{\ell}}}
   &=&\hB01\ln\frac{\tau_{n+e_\ell}}{\tau_{n-e_\ell}}-2\J 1m\b\ell1,\label{eq:syn1}\\
  \pps\ell1\ln\frac{\tau_{n+e_{\ell}}}{\tau_{n-e_{\ell}}}
   &=&\hB\ell1\ln\frac{\tau_{n+e_\ell}}{\tau_{n-e_\ell}}+2\kappa_\ell\J mm\b\ell1.\label{eq:syn2}
\eea 
A direct substitution of these formulas, as well as the formulas (\ref{1st der}) and~(\ref{2nd der}), in
(\ref{C008}), leads, along $\LR$, for $\ell=1,\dots,p$, to
\bean
 \hB01\ln\frac{\tau_{n+e_\ell}}{\tau_{n-e_\ell}}-2\J 1m\b\ell1&=&
       -\frac1{F_\ell}\((\hB02+\delta_{1,m})\hB\ell1\ln\tau_n-2 \J1m^2\kappa_\ell N(b_1)\),\\
 \hB\ell1\ln\frac{\tau_{n+e_\ell}}{\tau_{n-e_\ell}}+2\kappa_\ell\J mm\b\ell1&=&
        \frac1{F_\ell}\((\hB\ell2-\kappa_\ell\delta_{1,m})\hB01\ln\tau_n\right.\\
        &&\left.-2\J 1m(\J mm\kappa_\ell N(b_1)+\p\ell{b_1}\ln\tau_n)\),
\eean
where $F_\ell:=-\hB01\hB\ell1\ln\tau_n\Big|_\LR-n_\ell\J 1m$ (see (\ref{2nd der})).  Eliminating from these
equations the term which contains $\frac{\tau_{n+e_\ell}}{\tau_{n-e_\ell}}$, which can be done by applying
$\hB\ell1$ to the first equation and $\hB01$ to the second equation, and using that these operators commute, we get
the single equation
\bean
  \lefteqn{\hB\ell1\(\frac{2\J 1m\b\ell1F_\ell -(\hB02+\delta_{1,m})\hB\ell1\ln\tau_n\Big|_\LR
                +2\J1m^2\kappa_\ell N(b_1)}{F_\ell}\)}\\
  &=&\hB01\(\frac{(\hB\ell2-\kappa_\ell\delta_{1,m})\hB01\ln\tau_n\Big|_\LR
        -2\J 1m(\J mm\kappa_\ell N(b_1)+\p\ell{b_1}\ln\tau_n\Big|_\LR)}{F_\ell}\).
\eean
Using the fact that the derivative of a ratio amounts to a Wronskian, by clearing the denominator, and writing
$\hB\ell2$ as $\hB\ell2=\p\ell{b_2}+\kappa_\ell\,\vr_m$ (see (\ref{E}) and (\ref{eq:hBs})) and using the formula for $F_\ell$, one can rewrite the latter
equation as
\bea\label{A40} \lefteqn {-\left\{\pl^{(\ell)}_{b_2}\hB01\ln\tau_n\Big|_\LR\,,\,F_\ell \right\}_{\hB01}} \\
  &=&\left\{(\hB02+\delta_{1,m}+2\J1m\b\ell1\hB01)\hB\ell1\ln\tau_n\Big|_\LR +2\J1m^2(n_\ell\b\ell1-\kappa_\ell
  N(b_1)),\,F_\ell\right\}_{\hB\ell1}\no\\
  &&+\left\{(\kappa_\ell(\vr_m-\delta_{1,m})\hB01-2\J1m\p\ell{b_1})\ln\tau_n\Big|_\LR -2\kappa_\ell\J 1m\J
  mmN(b_1),F_\ell\right\}_{\hB01}.\no \eea
Finally the integral $\tau_n$ (as in (\ref{def:tau_n})), but integrated over the full range $\BR$, equals (see the
Appendix)
\be\label{eq:int_over_R}
  \tau_n(\BR^m)\Bigl|_{\LR}=g_n(c)\,e^{-\frac{\J mm}2\sum^p_{\ell=1}n_\ell{\b\ell1}^2}
        \prod_{1\leqs i<j\leqs p}(\b j1-\b i1)^{n_in_j},
\ee
with $g_n(c)$ a function, depending on $c_1,\dots,c_{m-1}$ and $n$ only. Thus one has, restricted to $\LR$,
\bea\label{for:new}
  \ln\tau_n(\BE)\Bigl|_{\LR}&=&\ln\BP_n(\BE)+\ln\tau_n(\BR^m)\Bigl|_{\LR},\\
  \ln\tau_n(\BR^m)\Bigl|_{\LR}&=&-\frac{\J mm}2\sum^p_{\ell=1}n_\ell(\b\ell1)^2
                +\sum_{1\leqs i<j\leqs p}n_in_j\ln(\b j1-\b i1)+\ln g(c).\no
\eea
When (\ref{for:new}) is substituted in (\ref{A40}), a few terms will appear where $\ln\tau_n(\BR^m)$ is acted upon
by a differential operator. We derive the formulas which will be used. First, it is clear that
$\hB01\tau_n(\BR^m)=0$. Therefore, since $\left[\hB01,\hB\ell1\right]=0$,
\be
  F_{\ell}=-\hB01\hB\ell1\ln\tau_n(\BE)\Bigl|_{\LR}-n_\ell\J 1m=-\hB01\hB\ell1\ln\BP_n(\BE)-n_\ell\J 1m.
\ee
Also, using $\pl^{(\ell)}_{b_1}\b i1=\kappa_{\ell}-\dt_{\ell,i}$, valid for $i=1,\dots,p$, one computes
$$
  \pl^{(\ell)}_{b_1}\ln \tau_n(\BR^m)\Big|_{\LR}= 
  \J mm(n_\ell\b\ell1-\kappa_\ell N(b_1))-n_{\ell}\sum_{i\neq\ell}\frac{n_i}{\b\ell1-\b i1}
$$
and therefore, since $\hB01\ln\tau_n(\BR^m)=0$, and by (\ref{eq:hBs}) and (\ref{eq:J_ders})
\bean
 \lefteqn{(\hB02+\delta_{1,m}+2\J1m\b\ell1\hB01)\hB\ell1\ln\tau_n(\BR^m)\Bigl|_{\LR}}\\
  &=&  \(c_1\pp{}{c_1}+\delta_{1,m}+\delta_{1,m}\sum_{i=1}^{p-1}\b i1\ppb i1\)\p\ell{b_1}\ln\tau_n(\BR^m)\Bigl|_{\LR}\\
  &=&2\J 1m^2(\kappa_\ell N(b_1)-n_\ell\b\ell1)
\eean
and
\bean
  \lefteqn{(\kappa_\ell(\vr_m-\delta_{1,m})\hB01-2\J1m\p\ell{b_1})\ln\tau_n(\BR^m)\Bigl|_{\LR}
    =-2\J1m\p\ell{b_1}\ln\tau_n(\BR^m)\Bigl|_{\LR}}\\
    &&=-2\J 1m\(\J mm(n_\ell\b\ell1-\kappa_\ell N(b_1))-n_{\ell}\sum_{i\neq\ell}\frac{n_i}{\b\ell1-\b i1}\).
\eean
Substituted in (\ref{A40}), yields the identity 
\bea 
  \lefteqn {-\left\{\pl^{(\ell)}_{b_2}\hB01\ln\tau_n\Bigl|_{\LR}\,,\,F_\ell \right\}_{\hB01}} \no\\ 
  &=&\left\{(\hB02+\delta_{1,m}+2\J1m\b\ell1\hB01)\hB\ell1\ln\BP_n,\,F_\ell\right\}_{\hB\ell1}\no\\
   &&+\left\{(\kappa_\ell(\vr_m-\delta_{1,m})\hB01-2\J1m\p\ell{b_1})\ln\BP_n-C_\ell,F_\ell\right\}_{\hB01}.  
 \label{A41}\eea
This ends the proof of Proposition~\ref{lemma 2.4}.
\end{proof}
\smallskip

\noindent
For $\ell=1,\dots,p$, using the shorthand notation,
\begin{equation}\label{7.4.6}
  X_\ell=\pl^{(\ell)}_{b_2}\hB01\ln\tau_n\Bigl|_{\LR},\quad
  {H}_{\ell}:= \left\{H_{ \ell }^{(1)},F_{\ell}\right\}_{\hB01}-\left\{H_{ \ell }^{(2)}, F_{\ell}\right\}_{\hB\ell1}
\end{equation} 
and $':=\hB01$, the equations (\ref{A2}) become (taking into account $\sum_{\ell=1}^p\pl^{(\ell)}_{b_2}=0$)
$$
  \{X_\ell,F_\ell\}=H_\ell,~~~~1\leqs \ell\leqs p,\mbox{~~~with}~~\sum_{\ell=1}^p X_\ell=0.
$$
%

\begin{proposition}\label{prop:7.1}

  Given for $\ell=1,\dots,p$ functions $H_\ell$ and $F_\ell$, such that the Wronskian of the derivatives
  $F_1',\dots,F_ p'$ is non-zero, the system of ODE's
  $$
    \{X_\ell,F_\ell\}=H_\ell ,~~~~1\leqs \ell\leqs p
  $$ 
  subjected to the condition $\sum_{\ell=1}^p X_\ell=0$, has a unique solution $(X_1,\dots,X_p)$, where $X_\ell$ is
  given by
  {\tiny $$\hspace*{3.5cm} \tiny {\ell \atop \downarrow} $$}\vspace*{-.4cm}
  \begin{equation}  X_\ell=\frac{F_\ell}{D}
  \det \(
    \begin{array}{ccccccccccc}
      F_1' &F_2'&F_3'& \cdots&-G_1&\cdots&F_ p'\\
       F_1'' &F_2''&F_3''& \cdots&-G_2&\cdots&F_ p''\\
      \vdots&\vdots&\vdots&&\vdots&&\vdots\\
      F_1^{(p )} &F_2^{(p )}&F_3^{(p )}&\cdots&-G_p&\cdots&F_ p^{(p )}  \\
    \end{array}\). \label{7.4.7}
 \end{equation} 
In this formula,  $D$ is the Wronskian of the functions $F_1',\ldots,F_ p'$,
 $$D:=
 \det \(
  \begin{array}{ccccccccccc}
    F_1' &F_2'&F_3'& \cdots& &F_ p'  \\
    F_1'' &F_2''&F_3''& \cdots& &F_ p''  \\
    \vdots&\vdots&\vdots&&&\vdots&\\
    F_1^{(p )} &F_2^{(p )}&F_3^{(p )}&\cdots& &F_ p^{(p )}\\
  \end{array}
  \)\neq0
$$
and the $G_i$'s are defined inductively as
\begin{equation}\label{7.4.8}
  G_{i+1}=G_i'+\sum_{\ell=1}^p \frac{H_\ell F_{\ell}^{(i)}}{F_{\ell}^2},~~G_0=0,~~G_1=\sum_1^p\frac{H_\ell}{F_\ell}.
\end{equation}
Moreover
\begin{equation}\label{7.4.9}
  \det \(
  \begin{array}{ccccccccccc}
    F_1&F_2&F_3&\ldots&F_ p&G_0 \\
    F_1' &F_2'&F_3'& \ldots&F_ p'&G_1 \\
    F_1'' &F_2''&F_3''& \ldots&F_ p''&G_2 \\
    \vdots&\vdots&\vdots&&\vdots&\vdots\\
    F_1^{(p )} &F_2^{(p )}&F_3^{(p )}& \ldots&F_ p^{(p )}&G_p
  \end{array}
  \)=0.
\end{equation}
\end{proposition}
\begin{proof}
If $X_\ell$ is a solution of the equation $\{X_\ell,F_\ell\}=H_\ell$, subjected to the condition $\sum_{\ell=1}^p
X_\ell=0$, then its derivatives are given by
\begin{equation}
  X_\ell^{(i)}=G_{\ell,i}+X_\ell\frac{F_\ell^{(i)}}{F_\ell},\label{4.7.8}
\end{equation}
where for a fixed $\ell$, the $G_{\ell,i}$ are defined inductively as
\begin{eqnarray*}
  G_{\ell,0}:=0,~~~ G_{\ell,1} := \frac{H_{\ell }}{F_{\ell}},~\ldots,~G_{{\ell },i+1}:= G'_{\ell,i}+\frac{H_{\ell }F_{\ell }^{(i)}}{F_{\ell }^2}.
\end{eqnarray*}
Indeed, starting with (\ref{4.7.8}) and using $X_{\ell }'=\frac{1}{F_{\ell }}(H_{\ell }+X_{\ell }F_{\ell }')$, one
computes inductively
\begin{eqnarray*}
  X_\ell^{(i+1)}&=&G'_{\ell,i}+X_{\ell}'\frac{F_\ell^{(i)}}{F_{\ell}}+X_\ell\frac{F_\ell^{(i+1)}}{F_\ell}
        -X_{\ell}\frac{F_{\ell }'F_{\ell }^{(i)}}{F_{\ell }^2}\\
  &=& \(G_{\ell, i}'+ \frac{H_{\ell }F_{\ell }^{(i)}}{F_{\ell }^2}\)+ X_{\ell }\frac{F_{\ell }^{(i+1)}}{F_{\ell }}
   = G_{{\ell },i+1}+ X_{\ell }\frac{F_{\ell }^{(i+1)}}{F_{\ell }},
\end{eqnarray*}
establishing (\ref{4.7.8}). Summing up (\ref{4.7.8}) for $\ell$ from $1$ to $p$, one finds
$$0=
 G_i+\sum_{\ell=1}^p
 X_{\ell }\frac{F_{\ell }^{(i)}}{F_{\ell }}
,~~~\mbox{where}~~
 G_i:=\sum_{\ell=1}^pG_{{\ell },i}.
 $$
Then solving this linear system for the $X_\ell$'s, one finds the ratio (\ref{7.4.7}) above. Then using that
solution and expressing $\sum_{\ell=1}^{p}X_\ell=0$ establishes (\ref{7.4.9}) and thus the proof of Proposition
\ref{prop:7.1}.
\end{proof}

This enables us to make the following statement, remembering the operators $\hB\ell1$, with
${}'=\hB01=\sum_{i=1}^m\J 1i\pl_{E_i}$, and
$\p\ell{b_1}$ with $\sum_{\ell=1}^p \p\ell{b_1}=0$. 

\begin{theorem}\label{Theo:7.2}
The probability $\BP_n=\BP^A_n(c,\BE)$ as in (\ref{prob-model}), with the linear constraint
${\sum_{\ell=1}^{p}}\kappa_\ell\b\ell1=0$, with ${\sum_{\ell=1}^{p}}\kappa_\ell=1$, satisfies a non-linear PDE in
the boundary points of the subsets $E_1,\dots,E_m$ and in the target points $\b11,\dots,\b p1$, given by the
determinant of a $(p+1)\times (p+1)$ matrix
\begin{equation}
  \det \(
    \begin{array}{ccccccccccc}
      F_1&F_2&F_3&\ldots&F_ p&G_0\\
      F_1'&F_2'&F_3'& \ldots&F_ p'&G_1\\
      F_1'' &F_2''&F_3''& \ldots&F_ p''&G_2\\
      \vdots&\vdots&\vdots&&\vdots&\vdots\\
      F_1^{(p )} &F_2^{(p )}&F_3^{(p )}& \ldots&F_ p^{(p )}&G_p \\
    \end{array}
    \)=0,
\end{equation}
where the $F_\ell$,~$H_\ell^{(i)}$ and $C_\ell$ are given by in Proposition \ref{lemma 2.4} and the $G_\ell$
inductively by
%
\bean
  G_{\ell+1}&:=&G_\ell'+\sum_{i=1}^pF_i^{(\ell)}\(\hB01\frac{H_i^{(1)}}{F_i}-\hB i1\frac{H_i^{(2)}}{F_i}\), ~~~~G_0:=0.
\eean

\end{theorem}

{\medskip\noindent{\it Proof of Theorem \ref{Theo:1.1}:\/} } It follows immediately from Theorem \ref{Theo:7.2} by noticing that in the notation of (\ref{diff-op}), the $ \hB\ell i$ are expressed as 
$$
\hB\ell 1=\pl_\ell,~~~
 \hB0 2=-\vr_0.
 $$

{\medskip\noindent{\it Proof of Corollary \ref{Cor:1.2}:\/} } The simplification comes from the fact that for
one-time (i.e., $m=1$) the operators $\pl_0$ and $\pl_\ell$ differ by very little, namely:
$$
 \pl_0=-\pl_{_{\!E}},  ~~~\pl_\ell=\pl_b^{(\ell)}+\kappa_\ell \pl_{_{\!E}}, ~~~\vr=\vr_0=\vr_m.
$$
This means that the expression in brackets in the definition of $G_{i+1}$ in (\ref{F,G}) can be re-expressed as follows, 
$$
 \pl_0\frac{H_\ell^{(1)}}{F_\ell}- \pl_\ell\frac{H_\ell^{(2)}}{F_\ell}
      = \pl_{_{\!E}}\frac{-H_\ell^{(1)}-\kappa_\ell H_\ell^{(2)}+2\kappa_\ell b_\ell F_\ell}{F_\ell}
        -\pl_b^{(\ell)}\frac{H_\ell^{(2)}}{F_\ell}=\pl_{_{\!E}}\frac{\bar H_\ell^{(1)}}{F_\ell}- 
        \pl_b^{(\ell)}\frac{H_\ell^{(2)}}{F_\ell},
 $$ 
upon setting $\bar H_\ell^{(1)}:=-H_\ell^{(1)}-\kappa_\ell H_\ell^{(2)}+2\kappa_\ell b_\ell F_\ell $, which one
checks\footnote{Upon using the commutation relation $[\vr_{_{\!E}},\pl_{_{\!E}}]=-\pl_{_{\!E}}$ and
$\vr=\vr_0=\vr_m$.} to be the expression $\bar H_\ell^{(1)}$ announced in (\ref{F,G-one-time}) and one repeats the
proof of Proposition~\ref{prop:7.1} with ${}^\prime=\pl_E$ (instead of ${}^\prime=\pl_0=-\pl_E$) and $X_\ell\mapsto
-X_\ell$, ending the proof of Corollary \ref{Cor:1.2}.\qed

\section{Examples}\label{sect7}

\subsection{One target point at the origin}

In this case, $m=2$, $p=1$ and the diagonal matrix $A=0$.  
The matrix $\JR$ reads
$$
\JR=\frac1{1-c^2}\left(\begin{array}{cc}
                                     -1&-c\\
                                      -c&-1
                                      \end{array}\right)
                                      $$
                                      and one checks
$$
\kappa_0=-1,~b_1=0,~\kappa_1=1,~\pl_b^{(0)}=\pl_b^{(1)}=
\vr_b=0, ~~\vr_0=\vr_{_{\!E_1}}-c\frac{\pl}{\pl c},
 ~~\vr_2=\vr_{_{\!E_2}}-c\frac{\pl}{\pl c},$$
\be\pl_0=-\frac{1}{1-c^2}\left(\pl_{_{\!E_1}}+c\pl_{_{\!E_2}}\right)
,~~~~~\pl_1=\frac{1}{1-c^2}\left(c\pl_{_{\!E_1}}+ \pl_{_{\!E_2}}\right),~~~~~C_\ell=0
\label{del}\ee
So, for $\ell=1$, one has
\bea 
  F_1&=&-\pl_0\pl_1\log \BP_n+\frac{nc}{1\!-\!c^2}=
    \frac{ 1}{(1\!-\!c^2)^2}\left(\pl_{_{\!E_1}}\!+\!c\pl_{_{\!E_2}}\right)
\left(c\pl_{_{\!E_1}}\!+\! \pl_{_{\!E_2}}\right)\log \BP_n+\frac{nc}{1\!-\!c^2}
\no\\
H_1^{(1)}&=&- \vr_2\pl_0\log \BP_n= (\vr_{_{\!E_2}}-c\frac{\pl}{\pl c}) \frac{1}{1-c^2}\left(\pl_{_{\!E_1}}+c\pl_{_{\!E_2}}\right)\log \BP_n
\no\\
H_1^{(2)}&=&-\vr_0\pl_1\log \BP_n=- ( \vr_{_{\!E_1}}-c\frac{\pl}{\pl c} )
\frac{1}{1-c^2}\left(c\pl_{_{\!E_1}}+ \pl_{_{\!E_2}}\right)\log \BP_n 
\label{F,H}
\eea 
and thus
$$G_0=0,~~~G_1=
F_1\left(\pl_0\frac{H_1^{(1)}}{F_1}-
  \pl_1\frac{H_1^{(2)}}{F_1}\right)
=\frac{1}{F_1}\left(\left\{H_1^{(1)},F_1\right\}_{\pl_0}
-
 \left\{H_1^{(2)},F_1\right\}_{\pl_1}\right)
 $$
 leading to the PDE, with $\pl_0$ and $\pl_1$ as in (\ref{del}) and $H_1^{(j)}$ and $F_i$ as in (\ref{F,H}):  (see \cite{AvM-coupled} and \cite{MR2150191})
$$
 \det\left(\begin{array}{cc}F_1&G_0\\
                                     \pl_0 F_1 &G_1\end{array}\right)
= \left\{H_1^{(1)},F_1\right\}_{\pl_0}
-
 \left\{H_1^{(2)},F_1\right\}_{\pl_1} =0$$

\subsection{
 Target points with some symmetry}

Consider non-intersecting Brownian motions leaving from $0$ and  forced to $p$ target points at time $t=1$, with the only condition that the left-most and right-most target points are symmetric with respect to the origin, with $p-2$ intermediate target points thrown  in totally arbitrarily; this example will be used in section \ref{section:example}. 
It is convenient to rename the target points $\beta_1<\ldots<\beta_p$, as follows:
\be \begin{array}{ccccccccccc}
\tilde a&<& -\tilde c_1&< \ldots <&-\tilde c_{p-2}&<&-\tilde a
\\ \\
~n_{+ }& &  ~n_1           &  \ldots      & n_{p-2}  &  &~n_- 
\end{array}\label{targets}\ee
with the corresponding number of particles forced to those points at time $t=1$. 
Using the change of variables (\ref{change of variables}) from $\beta_i$'s to 
 \be
b=(b_1,...,b_p)=(a,-c_1,-c_2,...,-c_{p-2},-a),
\label{A-matrix}
 \ee
one is led to the diagonal matrix of the form\footnote{Note the $c_i$ have nothing to do with the couplings $c_i$ appearing in (\ref{prob-model}).}:
 \be
 A:= \diag\bigl(\overbrace{  a,\dots,  a}^{n_+},\overbrace{    -c_1,\dots,    -c_1}^{n_1},\dots, 
    \overbrace{   -c_{p-2},\dots,   -c_{p-2}}^{n_{p-2}},
       \overbrace{ -   a,\dots,  -a}^{n_-}\bigr)
\label{7.A}
 \ee
with the obvious constraint $\sum_{ 1}^p \kappa_i b_i=\frac{1}{2}a+\frac{1}{2}(-a)=0, $ as in (\ref{constraint}), and thus
$$
\kappa_1=\kappa_p=\frac12\mbox{~~and}~~\kappa_i=0\mbox{  for} ~2\leqs i\leqs p-1.
$$
%
 %
 %
Moreover, setting $c=(c_1,\ldots,c_{p-2})$, formulae (\ref{diff-op}) become
\be
\pl_b^{(1)}=\frac{1}{2}\left(-\frac{\pl}{\pl a}-\pl_c\right),\pl_b^{(p)}=\frac{1}{2}\left(\frac{\pl}{\pl
     a}-\pl_c\right),\pl_{b}^{(\ell)}=\frac{\pl}{\pl c_{\ell-1}},\quad 2\leqs\ell\leqs p-1,
\label{7.diff-op}\ee
and $
\vr=\vr_E 
      -a\frac{\pl}{\pl a}-\vr_c
     ; $ also set ${}^\prime =\pl_E$. 
Besides the renaming $n_1=n_+,~n_p=n_-$ and $n_k\mapsto n_{k-1}$ for $2\leqs k\leqs p-1$, already mentioned, one also has, referring to formulae (\ref{F,G-one-time}), the following renaming:
\bean
  & & F_1\mapsto F_+, ~~ F_p\mapsto F_-, ~~  F_k\mapsto F_{k-1},\mbox{~for }2\leqs k\leqs p-1,\\
  & &\bar H_1^{(1)}\mapsto H_+^{(1)},~~\bar H_p^{(1)}\mapsto H_-^{(1)},~ ~ 
          H_1^{(2)}\mapsto H_+^{(2)}, ~~ H_p^{(2)}\mapsto H_-^{(2)},\\
  & & \bar H^{(1)}_\ell\mapsto H^{(1)}_{\ell-1},~ ~  H^{(2)}_{\ell}\mapsto H^{(2)}_{\ell-1}, 
          \mbox{ for } 2\leqs \ell \leqs p-1.
 \eean 
%
Then, one checks from Corollary \ref{Cor:1.2}, formulae (\ref{F,G-one-time}),
that\footnote{In the formulae below (\ref{Expressions-1-time}), the constants $C_\pm$ and $C_\ell$ have the value: 
$$C_\pm= -n_{\pm}\left(\pm a\pm\frac{n_{\mp}}{a}+2\sum^{p-2}_{r=1}\frac{n_{r}}{\pm a+c_{r}}\right) 
 ~~\mbox{and}~~ C_\ell=2n_{\ell}\left(c_{\ell}+\frac{n_+}{c_{\ell}\!+\!a}+\frac{n_-}{c_{\ell}\!-\!a}
     +\sum^{p-2}_{r=1\atop{r\neq\ell}}\frac{n_r}{c_{\ell}\!-\!c_r}\right) 
$$}
for $1\leqs \ell\leqs p-2$ and for $\BP:=\BP_n^A(E)$, with $\vr=\vr_E 
      -a\frac{\pl}{\pl a}-\vr_c$ (as in (\ref{prob-model}) for $m=1$)%
{\footnotesize\bea  
    F_{\pm}&=&\frac{1}{2}(\mp\frac{\pl}{\pl a}-\pl_c+\pl_{_{\!E}})\pl_{_{\!E}}\ln \BP+n_{\pm},~~~~~~~~~~~~~~~~~~~~~~~~~~~~F_{\ell} =\!\frac{\pl}{\pl c_{\ell}}\pl_{_{\!E}}\ln \BP\!+\!n_{\ell,}
 \no\\
    H^{(1)}_{\pm}&=&\frac{1}{4}\left(-2\pl_{_{\!E}}\vr+(\vr +3)\left(\mp\frac{\pl}{\pl a}-\pl_c+\pl_{_{\!E}}\right)\right)\ln \BP 
    +C_\pm,~~  H^{(1)}_{\ell}=2\frac{\pl}{\pl c_{\ell}}\ln \BP\!+\!C_\ell, \no\\
       H^{(2)}_{\pm}&=&\frac{1}{2}(1-\vr\pm 2a\pl_{_{\!E}})(\mp\frac{\pl}{\pl a}-\pl_c+\pl_{_{\!E}})\ln \BP,~~   H^{(2)}_{\ell}=(1\!-
\!\vr\!-\!2c_{\ell}\pl_{_{\!E}})\frac{\pl}{\pl c_{\ell}}\ln \BP 
  %
   %
.\label{Expressions-1-time}\eea } 
%
%
%
In accordance with formulae (\ref{7.4.6}), adapted to the case $m=1$, one defines for later use: 
\bea
H_{\pm}&:=&\{H^{(1)}_{\pm},F_{\pm}\}_{\pl_E}-\{H^{(2)}_{\pm},F_{\pm}\}_{\tfrac12 (\mp \frac{\pl}{\pl a}-\pl_c)}
\no\\
H_{\ell}&:=&\{H^{(1)}_{\ell},F_{\ell}\}_{\pl_E}-\{H^{(2)}_{\ell},F_{\ell}\}_{\frac{\pl}{\pl c_\ell}}
\label{H's}\eea
and 
one checks that, with this notation (\ref{H's}) and upon decoding formula (\ref{F,G-one-time}) for the $G_k$'s, 
$$
G_{k+1}=\pl_{E}G_k+\frac{H_+F_+^{(k)}}{F_+^2}+\frac{H_-F_-^{(k)}}{F_-^2}
+\sum_{\ell=1}^{p-2}\frac{H_\ell F_\ell^{(k)}}{F_\ell^2},
$$
where $F^{(k)}$ is a shorthand for $(\pl_E)^kF$.  With these expressions in mind, $\BP:=\BP_n^A(E)$ satisfies the
(near-Wronskian) PDE (\ref{PDE}), i.e., 
%
\begin{equation}
  \det \(
    \begin{array}{ccccccccccc}
      F_+&F_-&F_1&\ldots&F_ {p-2}&G_0\\
      F_+'&F_-'&F_1'& \ldots&F_ {p-2}'&G_1\\
      F_+'' &F_-''&F_1''& \ldots&F_ {p-2}''&G_2\\
      \vdots&\vdots&\vdots&&\vdots&\vdots\\
      F_+^{(p )} &F_-^{(p )}&F_1^{(p )}& \ldots&F_ {p-2}^{(p )}&G_p \\
    \end{array}
    \)=0,
\label{7.9}\end{equation}

{\bf Special case}: For Brownian motions forced to $a$ and $-a$, without the intermediate points, the formula (\ref{7.9}) turns into the following determinant, with $F_{\pm}$ and $H^{(i)}_{\pm}$ as in (\ref{Expressions-1-time}), but with all $c$-partials removed: 
{\footnotesize\bean
\lefteqn{ F_+F_- \det \(
    \begin{array}{ccccccccccc}
      F_+&F_-&  G_0\\
      F_+'&F_-'&  G_1\\
      F_+'' &F_-''& G_2\\
    \end{array}
    \)}\\
    &=& F_+F_-\det \(
    \begin{array}{ccccccccccc}
      F_+&F_-&   0\\
      F_+'&F_-'&  \frac{H_+}{F_+}+\frac{H_-}{F_-}\\
      F_+'' &F_-''& \frac{H_+'}{F_+}+\frac{H_-'}{F_-}\\
    \end{array}
    \)\\
    &=&(H_+F_-+H_-F_+)\{F_+,F_-\}'- (H_+'F_-+H_-'F_+)     \{F_+,F_-\}=0.
\eean}

\section{Pearcey process with inliers}\label{section:example}

In this section, we consider non-intersecting Brownian motions leaving from $0$ and  forced to $p$ target points at time $t=1$, with the only condition that the left-most and right-most target points are symmetric with respect to the origin, with $p-2$ intermediate target points thrown  in totally arbitrarily, exactly as in section 7.1.
 The purpose of this section is to identify the critical process obtained by letting $n:=n_+=n_-\rg \iy$ and by rescaling $\tilde a$ and the $\tilde c_i$ accordingly, while keeping $n_1,\ldots,n_{p-2}$ fixed. This is the content of Theorem \ref{Th: main}.


{\medskip\noindent{\it Proof of Theorem \ref{Th: main}:\/} } The proof consists of letting $n=n_+=n_- \rg \iy$ in
the kernel (\ref{BM kernel}) and in the PDE (\ref{PDE}). In the proof, which requires several steps, we shall
restrict ourselves to $m=1$ (one-time), except for {\em Step 2}, which deals with the kernel.
\newline{\em Step 1}:  ~ {\em The PDE.}  The probability $\BP:=\BP_n^A(E)$ satisfies the
(near-Wronskian) PDE (\ref{7.9}); see section 7.2.


\noindent {\em Step 2}: {\em The scaling limit of the Brownian kernel}. 
  Non-intersecting Brownian motions leaving from $0$, such that $n_r$ particles are forced to $\beta_r$ at time $t=1$, are given by the kernel (\ref{BM kernel}), which is, in this instance, conveniently rewritten as 
   \bea 
%
\lefteqn{H^{(n)}_{t_k,t_\ell}(x,y;\tilde a,-\tilde c_1,\ldots,-\tilde c_{p-2},-\tilde a)dy}\no\\
&=&
 -\frac{dy}{2 \pi^2 \sqrt{(1-t_k)(1-t_\ell)}}
  \int_{\mathcal{C}} dV
\int_{\Gamma_{L}} dU ~\frac{1}{U-V}
\no\\
&&~\times~\frac{e^{ -\frac{ t_kV^2}{1-t_k}  +
\frac{2xV}{1-t_k}  -n_+\ln (V-\tilde a)-n_-\ln (V+\tilde a)}}
{e^{ \frac{-t_\ell U^2}{1-t_\ell} +\frac{2yU}{1-t_\ell}  
 -n_+\ln (U-\tilde a)-n_-\ln (U+\tilde a)}}
%
~~\prod_{r=1}^{p-2}\bigg(\frac{U+\tilde c_r}{V+\tilde c_r}\bigg)^{n_r}
 \no\\ \no\\
 &&-\left\{ \begin{array}{l}
            0, ~~~~~~~~~~~~~~~~~~~~~~~~~~~~~~~~~~~~~~
            \mbox{for}~~t_k\geqs
            t_\ell,\\  
             \frac{dy}{\sqrt{\pi (t_\ell-t_k)}}
            e^{-\frac{(x-y)^2}{t_\ell-t_k}}
            e^{\frac{x^2}{1-t_k}-\frac{y^2}{1-t_\ell}},~~~~~
            \mbox{for}~~t_k<t_\ell.
            \end{array}\right.    
  \label{7.BM kernel}\eea  
One then uses the same steepest descent method as for the case without inliers; the so-called steepest descent
$F$-function is the one (depending on $U$ or $V$) appearing in the exponential, with three consecutive derivatives
being $=0$ at the origin; the change of integration variables $U=U'(n/2)^{1/4}$ and $V=V'(n/2)^{1/4}$ then leads,
in the limit for $n=n_+=n_-\rg \iy$ about the saddle point, to the kernel (\ref{Pearcey-kernel}) (see for instance
\cite{TW-Pear} and in the asymmetric case \cite{AOvM}). So, the limit is \be \lim_{n\rg \iy}
H^{(n)}_{t_i,t_j}(\tilde x,\tilde y;\tilde a,\tilde c_1,\ldots,\tilde c_{p-2},-\tilde a)d\tilde y\Bigr| _{
\begin{array}{l} t_k =\frac12 +\frac{\tau_k}{4\sqrt{2n}}\\ \tilde x= \frac{X}{4(n/2)^{1/4}}\\ \tilde y=
\frac{Y}{4(n/2)^{1/4}}\\ \tilde a=\sqrt{n}\\ \tilde c_\ell= u_\ell \left(\frac n2\right)^{1/4}
   \end{array} 
   }
   = K^{\PR}_{\tau_i,\tau_j}(X,Y;~u_1,\ldots,u_{p-2})dY,\label{7.limit}\ee
where $K^{\PR}_{\tau_i,\tau_j}(X,Y;~u_1,\ldots,u_{p-2})$ is the Pearcey kernel with inliers (\ref{Pearcey-kernel}).

  \bigskip

  {\em Step 3}: {\em The scaling limit of the PDE}. 
%
  %
%
As mentioned, for the proof we limit ourselves to the one-time case, i.e., $m=1$. 
  %
  %
  %
  %
  %
%
%
%
We now proceed in two steps: 
\newline {\bf (i)} The change of variables (\ref{change of variables}) (especially footnote 6) from the non-intersecting Brownian motion probability to the matrix model (\ref{prob-model}); this change of variables appears in the first column of the table (\ref{7.scaling}) below. In other terms, it is the time-dependent change from the variables $(\tilde x,\tilde a,\tilde c)$ to the variables $(x,a,c)$, yielding in particular the diagonal matrix $A$ as in (\ref{7.A}). %
 \newline {\bf (ii)} Subsequently apply the scaling given by (\ref{7.limit}) with $z:=n^{-1/4}$ and a very small renaming $s:=\tau/\sqrt{8},~v_j:=2^{1/4}u_j,~\xi:=X/2^{5/4}$ for computational convenience. This appears in the second column of table (\ref{7.scaling}) below. 
   \be\begin{array}{ll   |   llll}
 & & t  &=\frac12  +\frac{\tau}{4\sqrt{2n}}=\frac12\left(1+ \bigl(\frac{\tau}{\sqrt{8}}\bigr)z^2\right)
 =:\frac12\left(1+ s_kz^2\right) \\
 x&=\tilde x\sqrt{\frac{2 }{t(1-t)}}&\tilde  x&= \frac{X}{4(n/2)^{1/4}}=\bigl(\frac{X}{2^{5/4}}\bigr)\frac{z}{\sqrt{2}}
   =:\frac{\xi z}{\sqrt{2}}\\
a&=\tilde a\sqrt{\frac{2t }{ 1-t}}  & \tilde a&=\sqrt{n} =\frac{1}{z^2}
 \\
 c_\ell&=\tilde c_\ell \sqrt{\frac{2t }{ 1-t}}    &     \tilde c_\ell&= u_\ell \left(\frac n2\right)^{1/4}=\frac{(u_\ell 2^{1/4})}{\sqrt{2}z}
  =:\frac{v_\ell }{\sqrt{2}z}
 \end{array} \label{7.scaling}\ee

  Concatenating these two scalings leads to the following; in the string of equalities below, the change corresponding to {\bf (i)} is indicated by $\stackrel{*}{=}$, whereas the second change {\bf (ii)} is indicated by $\stackrel{**}{=}$:  
%
%
%
%
\bean 
   P (E,s,v) &:=& \ln\BP_n^{(  \tilde a,-   \tilde c_2,\ldots,
   - \tilde  c_{p-1},-  \tilde a)}(\mbox{all}~x_i(t)\in    \tilde E)\Bigr|_{\begin{array}{l}
    t=\frac{1}{2}(1+sz^2)\\
      \tilde  a= 1/z^2\\
     \tilde  c_i=v_i/(\sqrt{2}z)\\
      \tilde  E=  Ez/ \sqrt{2} \end{array}}
 \eean     
 \bea   &\stackrel{*}{=}&
     \ln  \BP_n^{A}\left(    \tilde E \sqrt{\frac{2}{t(1-t)}}
  ;  \overbrace{\sqrt{\frac{2t}{1-t}}(  \tilde a,    \tilde c,- \tilde a)
  }^{\mbox{\footnotesize entries of diagonal matrix $A$}}
  \right)\Bigr|_{\begin{array}{l}
    t=\frac{1}{2}(1+sz^2)\\
      \tilde  a= 1/z^2\\
     \tilde  c_i=v_i/(\sqrt{2}z)\\
      \tilde E=  Ez/ \sqrt{2} \end{array}}
   \no\\&\stackrel{**}{=}&\ln \BP^A_n  \left(\frac{2zE}{\sqrt{1\!-\!s^2z^4}};
  \overbrace{   \sqrt{\frac{1\!+\!sz^2}{1\!-\!sz^2}}
   \Bigl(\frac{\sqrt{2}}{z^2}
     ,\frac{v_i}{z}
     ~,
      -\frac{\sqrt{2}}{z^2}
    \Bigr)}^{\mbox{\footnotesize entries of diagonal matrix $A$}}
	     \right)
     \no\\&=: &\ln \BP^A_n  (E';\underbrace{a,c,-a}_{\begin{array}{c}\mbox{\footnotesize entries of  diagonal}\\
     \mbox{\footnotesize matrix $A$}\end{array}})=: Q  (E'; a,c).
 \label{7.comp} \eea 
  Note that in the rest of this section, $E$ and $E'$ refer to complement of compact intervals; i.e., we shall be dealing with gap probabilities. The identity (\ref{7.comp}) suggests the $z$-dependent map:
   $$ T_z^{-1}:~~(E,s,v_j)\mapsto (E',a,c_j) , 
 ~~~~1\leqs j\leqs p-2, 
 $$
given by
\be
  E'
  =
   \frac{2zE}{\sqrt{1-s^2z^4}},
   ~~a
  =\frac{\sqrt{2}}{z^2}
     \sqrt{\frac{1+sz^2}{1-sz^2}}
  ,~
   c_j=
    \frac{v_j}{z}
     ~\sqrt{\frac{1+sz^2}{1-sz^2}},
\label{map}\ee
with inverse map
$$ T_z:~~(E',a,c)\mapsto  (E,s,v_j),
~~~~1\leqs j\leqs p-2,
 $$ 
  given by
 \be E=\frac{\sqrt{2}azE'}{ a^2z^4+2 },~~
 s=\frac{a^2z^4-2}{z^2(a^2z^4+2)},~~
  v_j=\frac{\sqrt{2}c_j}{az}.
  \label{inverse map}
  \ee
  Then summarizing the above, one has 
  \bean
   Q (E'; a,c)&:=&\log \BP^A_n  (E';a,c,-a)=\log   \BP^A_n(T_z^{-1}(E;s,v_j))=:P (E;s,v)
   ,\eean
and thus
$$
  Q(E',  a,c)=P\left( \frac{\sqrt{2}azE'}{ a^2z^4+2 }, ~\frac{a^2z^4-2}{z^2(a^2z^4+2)},~
   \frac{\sqrt{2}c_j}{az}\right)
$$
satisfies the PDE (\ref{PDE}) in the variables $E',a,c$, in terms of the operators specified in (\ref{7.diff-op}), 
 with $F_\pm,~F_\ell, ~H^{(i)}_{\pm},~H^{(i)}_{\ell}$ given by (\ref{Expressions-1-time}). In order to express the PDE in terms of the function $P(E;s,v)$, one must express all partials of $Q(E';a,c)$ in terms of partials of $P(E;s,v)$ in $E, ~s,~v$; e.g.,
\bean
\pl_{E'}Q(E'; a,c)\Bigr|_{T_z}
&=&
 \frac{\sqrt{2}azE'}{ a^2z^4+2 }\pl_EP\Bigr|_{T_z} 
 =\frac{\sqrt{1-s^2z^4}}{2z}\pl_EP(E;s,v)
\eean
and thus the operators $\pl_{_{\!E'}}$ and $\pl_{_{\!E}}$, as acting on $Q$ and $P$ respectively, and similarly for the others, are related by the following; we also indicate what the relationship becomes for $z\rg 0$:   \footnote{Since $\pl_c=\sum_1^{p-2} \pl_{c_i}$ and $\pl_v=\sum_1^{p-2} \pl_{v_i}$, the third relation is valid for $\pl_c$ and $\pl_v$ as well.}
\bea
 \pl_{_{\!E'}} \Bigr|_{T_z}
&=&\frac{\sqrt{1-s^2z^4}}{2z}\pl_{_{\!E}}
=\left(
  \frac{1}{2z}-\frac14 s^2z^3-\frac1{16}s^4z^7+O(z^9)
\right)\pl_{_{\!E}}
 \no\\
 \vr_{_{\!E'}}\Bigr|_{T_z}&=&\vr_{_{\!E}}
\no\\
 \pl_{c_i}\Bigr|_{T_z}&=&
 z \sqrt{\frac{1-sz^2}{1+sz^2}}\pl_{v_i}
 =\left(z-sz^3+\frac12  s^2z^5+O(z^7)\right)\pl_{v_i}
\no\\
\sqrt{2} \frac{\pl}{\pl a}\Bigr|_{T_z}&=& {(1-sz^2)^2} 
  \sqrt{\frac{1+sz^2}{1-sz^2}}\frac{\pl}{\pl s}
-  { z^2} \sqrt{\frac{1-sz^2}{1+sz^2}}(\vr_v+sz^2\vr_{_{\!E}}) 
 \no\\
 &=& \frac{\pl}{\pl s}- {z^2} (\vr_v+s\frac{\pl}{\pl s})
- {sz^4} \left(\frac12 s\frac{\pl}{\pl s}+\vr_{_{\!E}}-\vr_v\right)+O(z^6)
. \label{del's}
\eea

For notational simplicity, derivatives will often be abbreviated in the obvious way:
\be
 (\pl_{_{\!E'}})^jF_i  \mapsto  F_i^{(j)}  ,~~
  (\pl_{_{\!E}})^jP  \mapsto P^{(j)} ,~~
  \frac{\pl}{\pl s} P  \mapsto  \dot P,\ldots,
 \label{abrev}\ee
 while keeping in mind from (\ref{del's}) that $\pl_{_{\!E'}}$ acting on functions of $(E',a,c)$, as $F_\pm, ~H_\ell$ and $G_\ell$, translates, {\em to leading order}, into $\pl_{_{\!E}}/(2z)$ acting on functions of $(E,s,v)$; also notice the {\em big gaps} in the first few terms of the series for $\pl_{_{\!E'}}$.  In view of the PDE (\ref{PDE}), one needs the series expansion in $z$ of the $F$'s, the $H$'s and the $G$'s and their derivatives $\pl_{E'}$. This is the content of: 
 
 \begin{lemma} \label{Lemma 7.4}Introducing the expression $\BY$, with $\vr=\vr_E-\vr_v,$ and $v=(v_1,\ldots,v_{p-2})$,
  \be
 \frac12\BY:=
 4(\vr-2s\pp{}{s}-2)\pl_{_{\!E}}^2P+16\pl_v\pl_{_{\!E}} { \dot  P} +8P^{\!\!\!\!\dot{{}}~\!\dot{}~\!\dot{}}+ \left\{\pl_{_{\!E}}   \dot P 
 ,\pl_{_{\!E}}^2 P\right\}
 _{\pl_{_{\!E}}},
\label{Y}\ee
one checks, (remember from (\ref{H's}) the definition of $H_\pm$ and $H_\ell$)
$$
\pl_{E'}^iF_\pm=\left(\frac{\pl_E}{2z}\right)^i\left(\frac{1}{z^4}+\frac{1}{8z^2}\pl_E^2P\mp\frac{1}{4\sqrt{2}z}\pl_E\dot P\right)+O(z^{-i})
$$
$$
\pl_{E'}^i F_\ell=\left( \frac{\pl_E}{2z}\right)^i\left(
\frac12 \pl_{v_\ell} \pl_E P +n_\ell-\frac{sz^2}{2}\pl_{v_\ell}\pl_EP
\right)+O(z^{3-i})
$$
$$
\frac{H_+}{F_+}+\frac{H_-}{F_-}+\sum_{\ell=1}^{p-2}\frac{H_\ell}{F_\ell}
\stackrel{*}{=}
\frac{1}{64z^2}\left(\BY-3(\pl_E^2P)(\pl_E^2\dot P)\right)+O(1)
$$
\be
\frac{H_+\pl_{E'}^i F_+}{F_+^2}+
\frac{H_-\pl_{E'}^i F_-}{F_-^2}+\sum_{\ell=1}^{p-2}\frac{H_\ell\pl_{E'}^i F_\ell}{F_\ell^2}
\stackrel{*}{=}
\frac{3}{32z}\left( \pl_E^3P\right) \left(\frac{\pl_E}{2z}\right)^{1+i}\dot P  +O(z^{-i-1})
 \label{der:F,G}\ee
and also, for $k=0,1,\ldots,$ one has
 \be G_{k+1}+\frac{3\sqrt{2}}{16}   (F_--F_+)^{(k+1)}P''
 = \frac{\BY^{(k)}}{16(2z)^{k+2}}+O(z^{-k-1}).\label{G's}\ee
 \end{lemma}
 
 \proof The formulae (\ref{der:F,G}) are straightforward computations; one of them involves the expression $\BY$ introduced in (\ref{Y}). The big gaps in the series (\ref{del's}) of  $\pl_{_{\!E'}}$ is responsible for the mere action of $\left( {\pl_E}/{2z}\right)^i$, in computing higher derivatives. 
 Moreover, in the third formula, one notices that the sums $\sum_{\ell=1}^{p-2}H_\ell/F_\ell$ on the left hand side of $\stackrel{*}{=}$ actually do not play any role in the leading terms, because $H_\ell$ and $F_\ell$ both are $O(1)$. Formula (\ref{G's}) is shown by induction; namely for $k=0$, one checks, using the formulae (\ref{der:F,G}),
 \bean
\lefteqn{G_1+\frac{3\sqrt{2}}{16}(F_--F_+)'P''}\\
&=&\sum^p_{\ell=1}\frac{H_{\ell}}{F_{\ell}}+\frac{3\sqrt{2}}{16}(F_--F_+)'P''\\
&=&\frac{1}{64z^2}(\BY-3P''\dot
P'')+\frac{3\sqrt{2}}{16z^2}\left(\frac{1}{4\sqrt{2}}\dot
P''P''\right)+{\bf
O}(1)
=\frac{\BY}{64z^2}+{\bf O}(1). \eean 
Assume inductively 
\be
 G_i+\frac{3\sqrt{2}}{16}   (F_--F_+)^{(i)}P''
=\frac{\BY^{(i-1)}}{16(2z)^{i+1}}+O(z^{-i})\mbox{   for } 1\leqs i\leqs k,
\label{induction}\ee
and prove it for $i=k+1$. 
Then, using the general definition (\ref{F,G-one-time}) of $G_{k+1}$ in terms of $G_k$, formula (\ref{induction}), the derivatives $\pl_{E'}$ of $F_\pm$ as in (\ref{der:F,G}) and the last formula of (\ref{der:F,G}), one checks
\bean
\lefteqn{
G_{k+1}+\frac{3\sqrt{2}}{16}   (F_--F_+)^{(k+1)}P''
}\\
&=&\pl_{E'}G_k+\frac{H_+F_+^{(k)}}{F_+^2}+\frac{H_-F_-^{(k)}}{F_-^2}+\sum_{\ell=1}^{p-2}\frac{H_\ell F_\ell^{(k)}}{F_\ell^2}
+\frac{3\sqrt{2}}{16}   (F_--F_+)^{(k+1)}P''\\
&=&\frac{\pl_E}{2z}\left(\frac{\BY^{(k-1)}}{16(2z)^{k+1}} -\frac{3\sqrt{2}}{16}   (F_--F_+)^{(k)}P''+O(z^{-k})\right)  \\
&&~~+\frac{H_+F_+^{(k)}}{F_+^2}+\frac{H_-F_-^{(k)}}{F_-^2}+\sum_{\ell=1}^{p-2}\frac{H_\ell F_\ell^{(k)}}{F_\ell^2}
+\frac{3\sqrt{2}}{16}   (F_--F_+)^{(k+1)}P''\\
&=&
 \frac{\BY^{(k)}}{16(2z)^{k+2}}+O(z^{-k-1}),
\eean
establishing Lemma \ref{Lemma 7.4}.\qed

 By Corollary \ref{Cor:1.2}, $Q  (E'; a,c)=\ln \BP^A_n  (E';  a,c)$ satisfies the PDE (\ref{7.9}), which induces a PDE for $P (E;s,v)=\ln   \BP^A_n(T_z^{-1}(E;s,v_j))
 $, remembering (\ref{7.comp}) and (\ref{map}). As pointed out, the PDE for $ Q  (E';  a,c)$  misses to be a Wronskian by the last column. It is appropriate to do some column operations; e.g., subtracting the first from the second and then adding the second, multiplied with $P''$, to the last one; also it is convenient to multiply the columns with $2$'s and $\sqrt{2}$'s. This gives us the determinant below, which vanishes according to Corollary \ref{Cor:1.2}. The second equality
 $\stackrel{*}{=}$
 uses in a straightforward way the series expansion of Lemma \ref{Lemma 7.4} above,
 %
{ \footnotesize
 \begin{equation}
0=\det \left(
  \begin{array}{lllllllllll}
   2F_+&\sqrt{2}(F_-\!-\!F_+)& 2F_1&\ldots& 2F_{p-2}&~G_0
   +\frac{3\sqrt{2}}{16}(F_--F_+)P''\\
   2F_+' &\sqrt{2}(F_-\!-\!F_+)'& 2F_1'& \ldots& 2F_{p-2}'&~G_1
  +\frac{3\sqrt{2}}{16}(F_-\!-\!F_+)'P'' \\
     2F_+'' &\sqrt{2}(F_-\!-\!F_+)''& 2F_1''& \ldots& 2F_{p-2}''&~G_2
    +\frac{3\sqrt{2}}{16}(F_-\!-\!F_+)''P'' \\
    \vdots&\vdots&\vdots&&\vdots&~\vdots
    \\
     2F_+^{(p )} &\sqrt{2}(F_-\!-\!F_+)^{(p )}& 2F_1^{(p )}&
    \ldots& 2F_{p-2}^{(p )}&~
    G_p+\frac{3\sqrt{2}}{16}(F_-\!-\!F_+)^{(p)}P'' \\
    %
    %
    %
\end{array}
\right) ,\no\end{equation}}

{\footnotesize
\begin{equation*}
\stackrel{*}{=}\det \left(
  \begin{array}{ccccccclllllllllll}
   \frac{2}{z^{4}}+\frac{  P''}{ (2z)^2}+O(\frac{1}{z})&
   \frac{\dot P'}{  (2z)}+O(z)&
   \frac{\pl P'}{\pl v_1}+2n_1+O(z^2)&\ldots
    \\
    \frac{  P'''}{ (2z)^3}+O(\frac{1}{z^2})
    &\frac{\dot P''}{  (2z)^2}+O(1)&
   \frac{1}{ (2z)}\frac{\pl P''}{\!\!\pl v_1}+O(z)& \ldots
  \\
     \frac{ P^{iv}}{ (2z)^4}+O(\frac{1}{z^3})
      &\frac{\dot P'''}{  (2z)^3}+O(\frac{1}{z})
      &\frac{ 1}{ (2z)^2}\frac{\pl P_n'''}{\!\!\pl v_1}+O(1)&
      \ldots
    \\
    \vdots&\vdots&\vdots&&
    &
    \\
    \frac{ P^{(p+2)}}{ (2z)^{p+2}}+O(\frac{1}{z^{p+1}}) &
    \frac{\dot P^{(p+1)}}{  (2z)^{p+1}}+O(\frac{1}{z^{p-1}})&
     \frac{1 }{ (2z)^{p}}\frac{\pl P^{(p+1)}}{\!\!\pl v_1}+O(\frac{1}{z^{p-2}})& \ldots
    \\
   \end{array}
\right.  
\no\end{equation*} }

{
\begin{equation}
   \left.
  \begin{array}{cccccccc}
 && \ldots& \frac{\pl P'}{\pl v_{p-2}}\!+\!2n_{p\!-\!2}
   +O(z^2)
    &  \frac{3}{16(2 z)}\dot P'P''+O(z)\\
 &&    \ldots&\frac{ 1}{ (2z)}\frac{\pl P''}{\!\!\pl v_{p-2}}+O(z)
     &\frac{\BY}{16 (2z)^2}+O(1) \\
&&     \ldots&\frac{ 1}{ (2z)^2}\frac{\pl
P'''}{\!\!\pl v_{p-2}}+
     O(1)&\frac{\BY'}{16 (2z)^3}+O(\frac{1}{z}) \\
 & &     & \vdots
    &\vdots&&
    &
    \\
 &&    \ldots&\frac{ 1}{ (2z)^{p}}
     \frac{\pl P^{(p+1)}}{\!\!\pl v_{p-2}}+
     O(\frac{1}{z^{p-2}})&
    \frac{\BY^{(p-1)}}{16 (2z)^{p+1}}+O(\frac{1}{z^{p-1}}) \\
    %
    %
    %
\end{array}
\right) 
\no\end{equation} }

$$
\hspace*{-0.5cm}\stackrel{**}{=}
\frac{C}{z^{6+p(p+1)/2}} {\cal W}_p\left(
 \pl_E^2\frac{\pl}{\pl s} P,\frac{\pl}{\pl v_1}\pl_E^2 P,
  \ldots,\frac{\pl}{\pl v_{p-2}}\pl_E^2 P,
   \BY\right)+O(\frac{1}{z^{4+p(p+1)/2}}).$$
   The last equality $\stackrel{**}{=}$ stems from the fact
   that the matrix consists of columns with increasing
   powers in $1/z$, except for the element $(1,1)$,
   whose leading term is $2/z^4$. Therefore the leading
   contribution of the determinant of the matrix will be given by
   $$
   \frac{2}{z^4}\times \mbox{ the determinant of the $(1,1)$-
   minor},
    $$
which indeed leads to equality $\stackrel{**}{=}$.  Also the term $-8s\pl_E^2\dot P_n$ could be removed by adding
$8s\times$ (the first column); but we prefer not to do this, in view of the conjecture \ref{conjecture}.  Taking the limit, when $n\rg \iy$, leads to the PDE for $P=\lim\log \BP_n$.  In
the end, one must undo the slight renaming (\ref{7.scaling}) of the variables $s=\tau/\sqrt{8},~v_j=2^{1/4}u_j,~x_i=\xi_i/2^{5/4}$ and go back to the $(\tau,u_j,\xi)$-variables, yielding 
$ 
\frac12 \BY=8^{3/2}\BX$, with $\BX$ as defined in (\ref{X}). This yields PDE (\ref{Wronskian_2}), which ends the proof of Theorem \ref{Th: main}.\qed

{\medskip\noindent{\it Very sketchy Proof of Corollary \ref{cor:no inl}:\/} } A detailed proof appears in Adler-Orantin-van Moerbeke \cite{AOvM}. In the absence of inliers ($p=2$), the Wronskian (\ref{Wronskian_2}) is the determinant of a $2\times 2$ matrix:
\be
0= {\cal W}_2\left[
     \pl_{_{\!E}}^2 \pl_\tau \ln \BP^{\cal P},
      ~\BX
 \right]_{\pl_{_{\!E}}}=\left\{ \pl_{_{\!E}}^2 \pl_\tau  \ln \BP^{\cal P},~\BX \right\}_{\pl_{_{\!E}}}
 .\label{Wr1}\ee
 Performing the same scaling limit on an {\em asymmetric} situation, with $2nq$ particles forced to $-\sqrt{n}$ and $2n(1-q)$ particles forced to $\sqrt{n}$ for $0<q<1$, with $q\neq 1/2 $,  leads to a PDE for the leading term having the form 
\be
  \left\{ \pl_{_{\!\!E}}^3    \ln \BP^{\cal P},~\BX \right\}_{\pl_{_{\!E}}}
 =0.\label{Wr2}\ee
 Thus  $\ln \BP^{\cal P} $ satisfies two different PDE's, (\ref{Wr1}) and (\ref{Wr2}), given by two Wronskians of $\BX$ with $\pl_{_{\!E}}^2 \pl_\tau \ln \BP^{\cal P}$ and $\pl_{_{\!\!E}}^3    \ln \BP^{\cal P}$. Then a functional-theoretical argument explained in \cite{AOvM} implies $\BX=0$.\qed
For inliers, we further conjecture -in analogy with the result in Corollary \ref{cor:no inl}- the validity of equations (\ref{Xu}) and (\ref{Xu1}), as stated in Conjecture \ref{conjecture}.

\section{Appendix: evaluation of the integral over the full range}
In this section we prove formula (\ref{eq:int_over_R}), i.e., we show that
\bean
  \lefteqn{\int_{\BR^{mN}}\Delta_N(v_1)\prod_{\ell=1}^p\(\frac{\Delta_{n_\ell}(\v\ell m)}{n_\ell!}\prod_{i=1}^{n_\ell}
    e^{-\frac12\sum\limits_{k=1}^{m}{{\v\ell{k;i}}^2}+\sum\limits_{k=1}^{m-1}{c_k\v\ell{k;i}\v\ell{k+1;i}}+{\b\ell1\v\ell{m;i}}}\)
    \prod_{k=1}^m dv_k}\hskip3cm\\
&=&g_n(c)\,e^{-\frac{\J mm}2\sum^p_{\ell=1}n_\ell{\b\ell1}^2}
        \prod_{1\leqs i<j\leqs p}(\b j1-\b i1)^{n_in_j},
\eean
with $g_n(c)$ a function, depending on $c_1,\dots,c_{m-1}$ and $n$ only, computed below. In view of the
representation of the above integral as the determinant of a moment matrix, as in (\ref{eq:tau_det}), it suffices
to prove that
\be
  \det
  \begin{pmatrix}
    M^{(1)}\\
    \vdots\\
    M^{(p)}
  \end{pmatrix}
  =
  g_n(c)\,e^{-\frac{\J mm}2\sum^p_{\ell=1}n_\ell{\b\ell1}^2}\prod_{1\leqs i<j\leqs p}(\b j1-\b i1)^{n_in_j},
  \label{eq:david}
\ee
where, for $\ell=1,\dots,p$, the $n_\ell\times N$ matrix $M^{(\ell)}$ is defined by 
\be
  M^{(\ell)}:=\left(\int_{\BR^m}w_1^jw_m^ie^{-\frac12\sum_{k=1}^m w_k^2+\sum_{k=1}^{m-1}c_kw_kw_{k+1}+\b\ell1 w_m}
    \prod_{k=1}^m dw_k\right)_{
     \renewcommand{\arraystretch}{0.8}
        \tiny{
          \begin{array}{c} 
            0\leqs i<n_\ell\\
            0\leqs j<N
          \end{array}}}.\no
\ee
Introducing for $a,b\in\BR$ the zero moment\footnote{Using
$$
 \int_{\BR^m}e^{-\frac{1}{2}\la Qw,w\ra+\la \ell,w\ra}
 dw_1\ldots dw_m
 =
  \frac{(2\pi)^{m/2}}{\sqrt{\det Q}}
e^{\frac{1}{2}\la Q^{-1}\ell,\ell\ra}
 ,$$
 for $Q:=-\J{}{}^{-1}$ and $\ell:=(a,0,\ldots,0,b)$.}
\bea
  m(a,b)&:=&\int_{\BR^m}e^{-\frac12(\sum_{k=1}^m w_k^2-2\sum_{k=1}^{m-1}c_kw_kw_{k+1})+aw_1+bw_m}
    \prod_{k=1}^m dw_k\no\\
  &=&(2\pi)^{m/2}\sqrt{-\det\J{}{}}\,e^{-\frac12(\J11 a^2+2\J 1mab+\J mmb^2)},\label{eq:mab}
\eea
we can express all the entries of $M^{(\ell)}$ as
\be
  M^{(\ell)}_{ij}=\frac{\partial^j}{\partial a^j}\,\frac{\partial^i}{\partial b^i}\, m(0,\b\ell1).
  \label{eq:mab_ders}
\ee
Let us first prove (\ref{eq:david}) in the case in which all $n_\ell$ are equal to $1$ (so that $p=N$). Then, it
follows from (\ref{eq:mab}) and (\ref{eq:mab_ders}) that, for $\ell=1,\dots,p$, the vector $M^{(\ell)}$ is, modulo
a constant which depends on $c_1,\dots,c_{m-1}$ and $N$, but not on $\ell$, of the form
\be
  M^{(\ell)}\sim e^{-\frac{\J mm}2{\b\ell1}^2}\left(1,\a_1(\b\ell1),\dots,\a_{N-1}(\b\ell1)\right),\no
\ee
where $\a_j(\b\ell1)$ is a polynomial in $\b\ell1$ of degree $j$, with leading term $\(-\J1m\b\ell1\)^j$, and whose
coefficients are independent of $\ell$, but depend on $c$. It follows that, if all $n_\ell$ are equal to $1$,
then\footnote{Note $g_N(c)=(-\J1m)^{N(N-1)/2}(-\det\J{}{})^{N/2}(2\pi)^{Nm/2},$ as easily follows from the
argument, and finally in the full case $g'_N(c)=\prod_{\ell=1}^{p}\prod_{k=1}^{n_\ell-1}k!g_N(c)$.}
\be \det
  \begin{pmatrix}
    M^{(1)}\\
    \vdots\\
    M^{(p)}
  \end{pmatrix}
  =
  g_N(c)\,e^{-\frac{\J mm}2\sum^p_{\ell=1}{\b\ell1}^2}\prod_{1\leqs i<j\leqs p}(\b j1-\b i1),
  \label{eq:app_p=N}
\ee
proving (\ref{eq:david}) in that case. Let us show how the other extreme case, where there is only one $n_\ell$ (so
that $p=1$ and $n_1=N$), is derived from it. Let $f:\BR\to\BR^N$ be a smooth function and let $\beta\in\BR$. Then
\be
  \det
  \begin{pmatrix}
    f(\beta)\\
    f'(\beta)\\
    \vdots\\
    f^{(N-1)}(\beta)
  \end{pmatrix}
  =
  \lim_{\beta_1,\dots,\beta_N\to\beta}
  \frac{\prod_{k=1}^{N-1} k!}{ \prod_{1\leqs i<j\leqs N}(\beta_j-\beta_i)}
  \det
  \begin{pmatrix}
    f(\beta_1)\\
    f(\beta_2)\\
    \vdots\\
    f(\beta_N)
  \end{pmatrix},
\ee
as follows by writing each $f(\beta_k)$ as a Taylor series around $f(\beta)$. Applied to 
\be
    f(\beta):=e^{-\frac{\J mm}2{\beta}^2}\left(1,\a_1(\beta),\dots,\a_{N-1}(\beta)\right),\no
\ee
and $\beta_1,\dots,\beta_N=\b\ell1,\dots,\b\ell N$ we conclude using (\ref{eq:mab_ders}) and (\ref{eq:app_p=N})
that, when $p=1$, then
\bean
  \det\(M^{(1)}\)&=&\lim_{\beta_1,\dots,\beta_N\to\b11}\prod_{k=1}^{N-1}k!\,g_N(c)e^{-\frac{\J mm}2\sum^N_{\ell=1}\beta_\ell^2}=g_n(c)e^{-\frac{\J mm}2N{\b11}^2},
\eean
proving (\ref{eq:david}) in this case. The proof of formula (\ref{eq:david}) in the intermediate case, when there
are several $n_\ell$, which are not equal to $1$, follows in a similar way from (\ref{eq:app_p=N}), taking the
limit $\beta_i\to \b j1$, for $i=1,\dots,N$, with $n_\ell$ of the $\beta_i$ going to $\b\ell1$, namely
$\beta_1,\dots,\beta_{n_1}\to \b11$, and $\beta_{n_1+1},\dots,\beta_{n_1+n_2}\to \b21$, and so on, where now one
divides by a product of $p$ Vandermonde determinants, each going with a collapsing group.

\def\cydot{\leavevmode\raise.4ex\hbox{.}}

\end{document}